\documentclass{article}
\usepackage[utf8]{inputenc}

\usepackage{fullpage}
\usepackage{amsfonts}
\usepackage{amsmath}
\usepackage{amssymb}
\usepackage{amsthm}
\usepackage{cite}

\usepackage{graphics}
\usepackage{graphicx}

\usepackage{hyperref}

\usepackage{subcaption}

\usepackage{latexsym,amsfonts,amssymb,amsmath,amsthm}

\usepackage{booktabs} 
\usepackage{graphicx}
\usepackage{pgfplots}
\usepackage[all]{nowidow}
\usepackage{soul}
\usepackage[utf8]{inputenc}
\usepackage{tikz}
\usetikzlibrary{automata}
\usepackage{multicol,comment}
\usepackage{algpseudocode,algorithm,algorithmicx}
\usepackage{cite}
\usepackage{hyperref}
\usepackage[inline]{enumitem} 

\usepackage{ulem}

\definecolor{blue}{HTML}{1F77B4}
\definecolor{orange}{HTML}{FF7F0E}
\definecolor{green}{HTML}{2CA02C}

\pgfplotsset{compat=1.14}

\newtheorem{tm}{Theorem}[section]
\newtheorem{prop}[tm]{Proposition}

\newtheorem{lem}[tm]{Lemma}

\newtheorem{rk}[tm]{Remark}

\numberwithin{equation}{section}
\numberwithin{tm}{section}
\newtheorem{Sim}[subsubsection]{Simulation}

\title{Competition-exclusion and coexistence  in a two-strain SIS epidemic model in patchy environments
}

\author{Jonas T. Doumat\`e\footnote{doumatt@yahoo.fr / jonas.doumate@fast.uac.bj; D\'epartement de Math\'ematiques (FAST), Institut de Math\'ematiques et de Sciences Physiques, Universit\'e d'Abomey-Calavi, Abomey-Calavi, Benin (Corresponding author)},\quad Tahir B. Issa\footnote{tahirbachar.issa@sjsu.edu; Department of Mathematics and Statistics, San Jose State University, San Jose, California, USA}, \quad   and \quad  Rachidi B. Salako\footnote{rachidi.salako@unlv.edu; Department of Mathematical Sciences, University of Nevada Las Vegas, Las Vegas, USA} }
\date{}

\begin{document}

\maketitle

\begin{abstract} This work examines the dynamics of solutions of a two-strain SIS epidemic model in patchy environments. The basic reproduction number $\mathcal{R}_0$ is introduced, and sufficient conditions are provided to guarantee the global stability of the disease-free equilibrium (DFE). In particular,  the DFE is globally stable when either: (i) $\mathcal{R}_0\le \frac{1}{k}$, where $k\ge 2$ is the total number of patches, or (ii) $\mathcal{R}_0<1$ and the dispersal rate of the susceptible population is large. Moreover,  the questions of competition-exclusion and coexistence of the strains are investigated when the single-strain reproduction numbers are greater than one. In this direction, under some appropriate hypotheses, it is shown that the strain whose   basic reproduction number and local reproduction function are the largest  always drives the other strain to extinction in the long run. Furthermore, the asymptotic dynamics of the solutions are presented when either both strain's local reproduction functions are spatially homogeneous or the population dispersal rate is uniform. In the latter case, the invasion numbers are introduced and the existence of coexistence endemic equilibrium (EE) is proved when these invasion numbers are greater than one. Numerical simulations are provided to complement the theoretical results.
    
\end{abstract}


\smallskip


\section{Introduction} 

\quad The novel human coronavirus disease 2019 (COVID-19) was first reported in the last trimester of 2019, and quickly spread around the world. The   emergence of different strains of the  disease  generated significant concerns due to  attendant waves after waves of infected populations across the world. As of March, 2020, the World Health Organization’s International Health Regulation Emergency Committee has declared  the COVID-19 outbreak a Public Health Emergency of International Concern.  In fact, in general,  the questions of developing and implementing  effective and adequate control strategies to alleviate the effects of infectious diseases, such as the COVID-19, on populations remain major concerns for public and health officials. These challenges are  related in part to the lack of resources and  good understanding of the dynamics of the disease.  They are also significantly influenced by  the unprecedented increase in  human migration rates, which has made the word more interconnected. Hence, the studies of mathematical models, which incorporate population movements and environmental heterogeneity, may help public and health officials to  take  informed decisions and implement safe disease control strategies. In the current work, we study the large time behavior of solutions of a two-strain susceptible-infected-susceptible (SIS) epidemic model in patchy environments and investigate how the parameters of the model affect its dynamics. In particular, our results reveal the important of both spatial heterogeneity and population movements on the dynamics of infectious diseases.

Consider the Susceptible-Infected-Susceptible (SIS) two-strain  epidemic model
\begin{equation}\label{model-eq1}
\begin{cases}
\frac{dS_j}{dt}=d_S\sum_{i\in\Omega}(L_{j,i}S_i-L_{i,j}S_j)-(\beta_{1,j}I_{1,j}+\beta_{2,j}I_{2,j})S_j+(\gamma_{1,j}I_{1,j}+\gamma_{2,j}I_{2,j}),  & j\in\Omega,\ t>0,\cr 
\frac{dI_{1,j}}{dt}=d_1\sum_{i\in\Omega}(L_{j,i}I_{1,i}-L_{i,j}I_{1,j})+\beta_{1,j}I_{1,j}S_{j}-\gamma_{1,j}I_{1,j}, & j\in\Omega,\ t>0,\cr  
\frac{dI_{2,j}}{dt}=d_2\sum_{i\in\Omega}(L_{j,i}I_{2,i}-L_{i,j}I_{2,j})+\beta_{2,j}I_{2,j}S_{j}-\gamma_{2,j}I_{2,j}, & j\in\Omega,\ t>0,\cr 
N=\sum_{j\in\Omega}(S_j+I_{1,j}+I_{2,j}),
\end{cases}
\end{equation}
where $N$ is the total number of individuals, $k\ge 2$ is the number of patches, and $\Omega=\{1,2,\cdots,k\}$.   For $j\in\Omega$ and $l\in\{1,2\}$:  $S_j$ is the total size of susceptible population on patch-$j$; $I_{l,j}$ is the total number of infected population with strain-$l$ on patch-$j$;  $d_S$ and $d_l$ are positive numbers and stand respectively for the dispersal rates of susceptible population and infected population with strain-$l$; $\beta_{l,j}$  is positive and represents the disease transmission rate resulting from interaction of the susceptible and infected population with strain-$l$ on patch $j$; $\gamma_{l,j}$ is positive and represents the recovery rate from strain-$l$ on patch-$j$.   Note that the $L_{j,j}$'s terms do not appear in \eqref{model-eq1}, so  we set $L_{j,j}=0$ for convenience. For $i, j\in\Omega$, $L_{i,j}$ is a nonnegative number and represents migration rate from  the patch-$j$ to the patch-$i$. Throughout this work, we shall suppose that the following standing assumption holds.

\medskip

\noindent{\bf (A1)} The matrix $L=(L_{i,j})_{i,j=1}^k$ is nonnegative, symmetric and irreducible.

\medskip

\noindent Assumption {\bf (A1)} indicates that the migration rates between any two patches are the same and individuals can move directly or indirectly from any patch to another.

When the dispersal rates are neglected in \eqref{model-eq1}, we obtain a particular type of a two-strain ODE-SIS  epidemic model studied by Bremermann and Thieme in \cite{BT1989}. An important conclusion reached in \cite{BT1989} for \eqref{model-eq1} when $d_S=d_1=d_2=0$ is that the pathogen strain that does not optimize the {\it basic reproduction} number dies out asymptotically.  This fact is known as the {\it competition-exclusion principle} in the literature.  The basic reproduction number measures the expected number of secondary cases caused by a single index in an otherwise susceptible population. Since this interesting work, several studies have been devoted to investigate the competition-exclusion and coexistence of different strains of an infectious disease.

In \cite{AA2003}, Ackleh and Allen examined the competition-exclusion and coexistence for pathogens in an SIR epidemic model with variable population size. In particular, under appropriate hypotheses on the parameters, they showed that several pathogens for a single host leads to exclusion of all pathogens except the one with the largest basic reproduction number. In \cite{AA2005}, Acklesh and Allen extended their study to  SIR and SIS epidemic models with multiple pathogen strains which assume total cross immunity, standard incidence, and density-dependent host mortality. Again, under some conditions on the parameters of these models, they established the competition-exclusion between multiple strains of the disease. For some related studies on multi-strain ODE models, we refer the interested readers to \cite{Allen2003, IML2005, BIS2014, MA2007, MCTM2018} and the references cited therein.

 Assume that $I_2=0$. Then system \eqref{model-eq1} reduces to a single-strain model, which  was recently studied by Li and Peng \cite{LP2023}. In this setting, they introduced the single-strain basic reproduction number $\mathcal{R}_{0,1}(N)$ (see formula \eqref{R_0-l} below) and  investigated  the  large-time behavior of solutions. Among other results, they established the global stability of the disease free equilibrium (DFE) if $\mathcal{R}_{0,1}(N)\le 1$ and either the local basic reproduction function $\mathfrak{R}_{1}(N)$ (see  formula \eqref{local-R_0} below) is constant or the population dispersal rate is uniform. However, when $\mathcal{R}_{0,1}(N)>1$, they proved that the disease is endemic and the single-strain model has at least one endemic equilibrium (EE). Furthermore, the work \cite{LP2023} established  the uniqueness of the single-strain EE  under some additional assumptions, which include:  (i)  the dispersal rate of the susceptible group is greater or equal to that of the infected group, (ii) transmission rate is patch independent,  and (iii) $\mathfrak{R}_{1}(N)$ is constant. The asymptotic profiles of the single-strain EE as either the dispersal rate of the susceptible or infected groups approximate zero were also studied in \cite{LP2023}. The results of \cite{LP2023} extend some known results (\cite{CS2022, DW2016, WJL2018, WZ2016}) on the continuous diffusive models to the patch model.

 Indeed, when dispersal movement of the population is assumed to occur locally and randomly  in  adjacent directions, the following PDE-SIS model 

 \begin{equation}\label{PDE-model}
     \begin{cases}
         \partial_tS=d_S\Delta S-(\beta_1I_1+\beta_2I_2)S+(\gamma_1I_1+\gamma_2I_2), & x\in\Omega,\ t>0,\cr 
         \partial_tI_1=d_1\Delta I_1 +\beta_1I_1S-\gamma_1I_1  & x\in\Omega,\ t>0,\cr 
         \partial_tI_2=d_2\Delta I_2 +\beta_2I_2S-\gamma_2I_2  & x\in\Omega,\ t>0,\cr 
         0=\partial_{\vec{n}}S=\partial_{\vec{n}}I_1=\partial_{\vec{n}}I_2 & x\in\partial\Omega,\ t>0,\cr 
         N=\int_{\Omega}(S+I_1+I_2),
     \end{cases}
 \end{equation}
 can be used to study the dynamics of the disease. In \eqref{PDE-model}, $\Omega$ is an open bounded domain in $\mathbb{R}^n$, $n\ge 1$, with a smooth boundary $\partial\Omega$. $\vec{n}$ denotes the  outward normal unit derivative on $\partial\Omega$. The parameters of \eqref{PDE-model} have the same meanings as those in the multiple patches model \eqref{model-eq1}. Hence, the ODE-SIS system \eqref{model-eq1} can be seen as a discrete in space of the continuous in space  PDE-SIS model \eqref{PDE-model}. The two-strain PDE-SIS model \eqref{PDE-model} was first studied by Ackleh, Deng and Wu \cite{ADW2016}, then by Salako \cite{SalakoJMB2023}, and recently by Castellano and Salako \cite{CS2023}. In these works, the authors defined the basic reproduction number of \eqref{PDE-model} and studied the asymptotic profiles of coexistence steady states, and  the large-time behavior of classical solutions. In particular, the authors of \cite{ADW2016} established the competition-exclusion of the strains when the local reproduction functions are spatially homogeneous. In \cite{SalakoJMB2023}, sufficient criteria for existence and non-existence of coexistence EE equlibria of \eqref{PDE-model} are obtained. Moreover,  the asymptotic profiles of coexistence EE solutions of \eqref{PDE-model} as the diffusion rates of some of the subgroups of the population converge to zero are established in \cite{SalakoJMB2023}.  The work \cite{CS2023} considered a more general model and also established the competition-exclusion of the strains if at least one strain local reproduction function is spatially homogeneous. Our results in the current manuscript examine the extent to which the findings of \cite{ADW2016} and \cite{CS2023} on the PDE-SIS model hold for the two-strain multiple patches model \eqref{model-eq1}. Furthermore, we  establish some new results for the two-strain multiple patches model \eqref{model-eq1}, which remain open for the PDE-SIS model \ref{PDE-model}. In particular, for system \eqref{model-eq1}, Theorem \ref{T1} below establishes the uniform persistence of the susceptible population. It also establishes the global stability of the DFE when either  the basic reproduction number is: (i) less or equal to the reciprocal of the total number of patches,  or (ii)  less than one and the dispersal rate of the susceptible population is sufficiently large.

 The current work also examines the global dynamics of solutions of the two-strain SIS multiple patches system \eqref{model-eq1} and discusses the extent to which the competition-exclusion principle hold. In particular Theorems \ref{T2}, \ref{T3}, and \ref{T4} establish the competition exclusion of the strains under quite general assumption on the parameters of the model. Theorem \ref{T5} introduces the invasion numbers and establishes the coexistence of the strains when these numbers are greater than one. 

There are  several  studies on the single-strain PDE-SIS epidemic model of system \eqref{PDE-model} (see \cite{CS2022, DW2016, WJL2018, WZ2016} and the references therein). Note that the force of infection used in the mathematical models \eqref{model-eq1} and \eqref{PDE-model} is $\beta_iI_i$, and  refers as the mass-action or density-dependent transmission mechanism.  Another popular transmission mechanism used in the modeling of infectious diseases is the standard or frequency-dependent infection mechanism, in which case the force of infection is given by $\frac{\beta_iI_i}{S+I_1+I_2}$.  For some results on the single-strain epidemic models with the standard transmission mechanism, we refer to \cite{Allen2007, Allen2008, CLL2017, GD2019, G2020, GL2021, LSP2023_JMB, P2009, Peng_Yi2013, SLX2018, WLR2019} and the references cited therein. See also \cite{LS2021, LS2023, WTM2017} for some recent progress on the multi-strain diffusive epidemic model with standard transmission mechanism.  

The rest of the paper is organized as follows. In section \ref{Sec2}, we first introduce some notations and definitions, and then state our main results. We complete this section with some numerical simulations and some discussion of our main results. Section \ref{Sec3} introduces some preliminary results, essential for the clarity of our presentations. The proofs of our main results are given in section \ref{Sec4}.

%

\section{Notations, Definitions and Main Results}\label{Sec2}

\subsection{Notations and Definitions} For convenience, we introduce a few notations and definitions. First, we would like to rewrite system \eqref{model-eq1} is a compact form. To this end, let $\mathcal{L}=(\mathcal{L}_{i,j})_{i,j=1}^{k}$ denote the square matrix with entries 
\begin{equation}
    \mathcal{L}_{i,j}=\begin{cases}-\sum_{l\in\Omega}L_{l,j} & i=j\in\Omega\cr 
    L_{i,j} & i\ne j\in\Omega.
    \end{cases}
\end{equation}
It follows from assumption {\bf (A1)} that the matrix $\mathcal{L}$ is  symmetric, irreducible and has nonnegative off diagonal  entries. For each $X\in\{S,I_{l},\beta_l,\gamma_{l}\}$, let $X$ denote the column vector in $\mathbb{R}^k$,  $X=(X_1,\cdots,X_k)^T$.
Given two column vectors $X$ and $Y$ in $\mathbb{R}^k$, define the Hadamard product $X\circ Y$,  $X\circ Y=(X_1Y_1,\cdots,X_kY_k)^T$, 

and  denote by ${\rm diag}(X)$ the diagonal matrix with diagonal entries $[{\rm diag}(X)]_{ii}=X_i$, $i=1,\cdots,k$.
Using the above notations, system \eqref{model-eq1} can be written as 

\begin{equation}\label{model-eq2}
    \begin{cases}
        \frac{d}{dt}S=d_{S}\mathcal{L}S +\sum_{l=1}^2(\gamma_l-\beta_l\circ S)\circ I_l &t>0,\cr 
        \frac{d}{dt}I_l=d_l\mathcal{L}I_l +(\beta_l\circ S-\gamma_l)\circ I_{l} & t>0, \ l=1,2,\cr 
        N=\sum_{j\in\Omega}(S_j+\sum_{l=1}^2I_{l,j})>0.
    \end{cases}
\end{equation}

Due to biological interpretations of the vectors $S$ and $I_{l}$, we will only be interested in nonnegative solutions of \eqref{model-eq2}. Let $\mathbb{R}_+$ denote the set of nonnegative real numbers.   Note that the right-hand side of \eqref{model-eq2} is locally Lipschitz on $\left[\mathbb{R}^k\right]^3$. Hence, given an initial data $(S(0),I_1(0),I_2(0))\in \left[\mathbb{R}_+^{k}\right]^3$,  \eqref{model-eq2} has a unique solution $(S(t),I_1(t),I_2(t))$ defined on a maximal interval of existence $(0,T_{\max})$. Observe that  $\mathcal{L}$ generates a strongly-positive matrix-semigroup $\{e^{t\mathcal{L}}\}_{t\ge 0}$ on $\mathbb{R}^k$. We have endowed $\mathbb{R}^k$ with the usual order induced by the cone of vectors with nonnegative entries $\mathbb{R}_+^{k}$. Hence, it follows from the comparison principle for cooperative systems that $(S(t),I_{1}(t),I_2(t))\in \left[\mathbb{R}^k_{+}\right]^3$ for every $t\in[0,T_{\max})$. A direct computation based on \eqref{model-eq1} gives
$$
\frac{d}{dt}\sum_{j\in\Omega}(S_{j}(t)+\sum_{l=1}^2I_{l,j}(t))=0\quad \forall\ t\in[0,T_{\max}).
$$
Hence, 
\begin{equation}\label{Eq1:1}
 \sum_{j\in\Omega}(S_{j}(t)+\sum_{l=1}^2I_{l,j}(t))=N\quad \forall\ t\in[0,T_{\max}),   
\end{equation}
from which we deduce that $T_{\max}=\infty$. Note also from \eqref{Eq1:1} that solution operator of \eqref{model-eq2} is globally bounded. Thanks to \eqref{Eq1:1} and the fact that the positive constant $N$ is fixed throughout the whole manuscript, the semiflow generated by  solutions of \eqref{model-eq1} leaves invariant the set 
$$
\mathcal{E}:=\Big\{(S,I_1,I_2)\in [\mathbb{R}^k_+]^3\ : \ \sum_{j\in\Omega}(S_j+I_{1,j}+I_{2,j})=N\Big\}.
$$    It is easy to see that $\mathcal{E}$ is a compact subset of $[\mathbb{R}^k_+]^3$, being closed and bounded. Note from \eqref{model-eq2} that if for some $l=1,2$, $I_{l,j}(0)=0$ for every $j\in\Omega$, then $I_{l,j}(t)=0$ for all $t>0$ and $j\in\Omega$. However, since $\mathcal{L}$ is irreducible with nonnegative off diagonal entries, if $I_{l,j_0}(0)>0$ for some $l=1,2$ and $j_0\in\Omega$, then  $I_{l,j}(t)>0$ for all $t>0$ and $j\in\Omega$. We say that a solution  $(S(t),I_1(t),I_2(t))$ has positive initial data if for each $l=1,2$,  $I_{l,j}(0)>0$ for some  $j\in\Omega$.

A vector $(S,I_{1},I_2)\in \left[\mathbb{R}_+^k\right]^3$ is an equilibrium solution of \eqref{model-eq2} if it satisfies the system of algebraic equation 

\begin{equation}\label{model-eq3}
    \begin{cases}
       0=d_{S}\mathcal{L}S +\sum_{l=1}^2(\gamma_l-\beta_l\circ S)\circ I_l, \cr 
        0=d_l\mathcal{L}I_l +(\beta_l\circ S-\gamma_l){\circ} I_{l} &  \ l=1,2,\cr 
        N=\sum_{j\in\Omega}(S_j+\sum_{l=1}^2I_{l,j}).
    \end{cases}
\end{equation}

Given a $k\times k$ square matrix $M$, we denote by $\lambda_{*}(M)$ its spectral bounds,
\begin{equation}
    \lambda_{*}(M):=\sup\{\mathfrak{R}e(\lambda)\ :\ \lambda\in\sigma(M)\},
\end{equation}
where $\mathfrak{R}e(\lambda)$ is the real part of $\lambda\in \mathbb{C}$, 
and by $\rho(M)$ its spectral radius,
$$
\rho(M):=\max\{|\lambda|\ :\ \lambda\in\sigma(M)\},
$$
where $\sigma(M)$ is the spectrum of $M$.
 
A disease free equilibrium (DFE) is an equilibrium of the form $(S,0,0)$.  Since the matrix $\mathcal{L}$ generates a strongly-positive semigroup and $\sum_{i\in\Omega}\mathcal{L}_{j,i}=0$ for each $j\in\Omega$, then by the Perron-Frobenius theorem, $\lambda_*(\mathcal{L})=0$. Moreover, $\lambda_*(\mathcal{L})$ is simple. For convenience, let ${\bf 0}$ denote the null vector in $\mathbb{R}^k$  and ${\bf 1}$ denote the $k$-column vector with entries ones, that is ${\bf 1}:=(1,\cdots,1)^T.$
 
 It is easy to see that column vector $\frac{N}{k}{\bf 1}$ is an eigenvector associated with $\lambda_*(\mathcal{L})$. As a result, we obtain that ${\bf E}^0:=(\frac{N}{k}{\bf 1},{\bf 0},{\bf 0})$ is the unique DFE of \eqref{model-eq2}. An equilibrium solution $(S,I_1,I_2)$ of \eqref{model-eq2} for which $I_{l,j}>0$ for every $j\in\Omega$ for some $l\in\{1,2\}$ is called an endemic equilibrium (EE) solution of \eqref{model-eq2}. An EE solution of the form $(S,I_1,0)$ (resp. $(S,0,I_2)$ ) is called a strain-1 (resp. strain-2) EE solution. An EE solution $(S,I_{1},I_2)$ for which $I_{l,j}>0$ for each $j\in\Omega$ and $l\in\{1,2\}$ is called a coexistence-EE solution.

 Linearizing \eqref{model-eq2} at ${\bf E}^0$, we get \begin{equation}\label{model-eq4}
    \begin{cases}
        \frac{d}{dt}P=d_{S}\mathcal{L}P +\sum_{l=1}^2(\gamma_l-\frac{N}{k}\beta_l)\circ Q_l &t>0,\cr 
        \frac{d}{dt}Q_l=d_l\mathcal{L}Q_l +(\frac{N}{k}\beta_l-\gamma_l)\circ Q_{l} & t>0, \ l=1,2,\cr 
        0=\sum_{j\in\Omega}(P_j+\sum_{l=1}^2Q_{l,j}).
    \end{cases}
\end{equation}
Note that for each $l=1,2$, the second equation in \eqref{model-eq4} decouples from the first equation.
 Clearly, for each $l\in\{1,2\}$, the square matrix $\mathcal{V}_l$ defined by \begin{equation}\label{V-l}
 \mathcal{V}_l:={\rm diag}(\gamma_l)-d_l\mathcal{L}
 \end{equation} is invertible. Thanks to \eqref{model-eq4}, for each $l\in\{1,2\}$,  setting 
 \begin{equation}\label{F-l}
 \mathcal{F}_l={\rm diag}(\beta_l),
 \end{equation} 
 and following the next generation matrix theory, strain-$l$'s basic reproduction $\mathcal{R}_{0,l}(N)$ is given by
 \begin{equation}\label{R_0-l}
     \mathcal{R}_{0,l}(N)=\frac{N}{k}\rho(\mathcal{F}_{l}\mathcal{V}_l^{-1}).
 \end{equation}
For each $l\in\{1,2\}$, it is well known that $\mathcal{R}_{0,l}(N)-1$ and $\lambda_*(\frac{N}{k}\mathcal{F}_l-\mathcal{V}_l)$ have the same sign (see \cite{Allen2007}). The basic reproduction $\mathcal{R}_0(N)$ of the two-strain model \eqref{model-eq2} is 
 \begin{equation}
     \mathcal{R}_0(N)=\max\Big\{\mathcal{R}_{0,l}(N)\ :\ l=1,2\Big\}=\frac{N}{k}\max\Big\{ \rho\big(\mathcal{F}_l\mathcal{V}_l^{-1}\big)\ :\ l=1,2\Big\}.
 \end{equation}
 For $l\in\{1,2\}$, it follows from \cite{LP2023} that \eqref{model-eq1} has a single-strain-$l$ EE solution if $\mathcal{R}_{0,l}>1$. The work \cite{LP2023} also provided sufficient conditions for the uniqueness and stability of single-strain EE of \eqref{model-eq1}. 

Given, $j\in\Omega$ and $l\in\{1,2\}$, the local basic reproduction on patch-$j$ of strain-$l$ is 
\begin{equation}\label{local-R_0}
    \mathfrak{R}_{l,j}(N)=\frac{N}{k}\mathfrak{R}_{l,j}\quad \text{where}\quad \mathfrak{R}_{l,j}
    :=\frac{\beta_{l,j}}{\gamma_{l,j}},
\end{equation}
and the two-strain local reproduction number is 
\begin{equation}
    \mathfrak{R}_j(N)=\max_{l=1,2}\mathfrak{R}_{l,j}(N).
\end{equation}
Note that $\beta_{l}=\mathfrak{R}_{l}\circ\gamma_{l}$ for each $l=1,2$, where $\mathfrak{R}_l=(\mathfrak{R}_{l,1},\cdots,\mathfrak{R}_{l,k})^T$. We also introduce the following sets
\begin{equation*}
    \Sigma_1:=\Big\{j\in{\Omega} : \mathfrak{R}_{1,j} > \mathfrak{R}_{2,j}\Big\},\ \Sigma_2:=\Big\{j\in {\Omega} : \mathfrak{R}_{2,j}> \mathfrak{R}_{1,j}\Big\} \ \text{and}\ \Sigma_0:= \Big\{j\in{\Omega}\ :\  \mathfrak{R}_{1,j} = \mathfrak{R}_{2,j}\Big\}.
\end{equation*}
The set $\Sigma_1$ represents the region of ${\Omega}$ where strain-1 is able to spread more quickly than strain-2. Conversely, $\Sigma_2$ is the region of $\Omega$ where strain-2 is more infectious than strain-1. The set $\Sigma_0$ is the region where strain-1 and strain-2 local reproduction functions are equal.

 We complete this section by introducing the high/low-risk patches.  Fix $j\in\Omega$ and $l\ne p=1,2$. Patch-$j$ is said to be high-risk (resp. low-risk) patch for strain-$l$ if the local reproduction number $\mathfrak{R}_l(N)$ is bigger (resp. smaller) than one on that patch, that is $\mathfrak{R}_{l,j}(N)>1$ (resp. $\mathfrak{R}_{l,j}(N)<1$). Hence the sets $H_l^{+}$ and $H_l^{-}$ given by 
 $$
 H_l^+:=\big\{j\in\Omega : \mathfrak{R}_{l,j}(N)>1\big\}\quad \text{and}\quad H_l^-:=\big\{j\in\Omega : \mathfrak{R}_{l,j}(N)<1\big\}
 $$
are called the strain-$l$'s high-risk region and low-risk region, respectively. Note that in the absence of movement and when the other strain is absent, strain-$l$ is endemic on any patch in its high-risk region while it goes extinct on any patch in its low-risk region. Hence, thanks to the competition exclusion-principle, when $d_S=d_1=d_2=0$, strain-$l$ is dominant and drives out strain-$p$ on any patch in $\Sigma_l\cap H_l^+$. When $\Sigma_1\cap H_1^+\ne\emptyset$ and $\Sigma_2\cap H_2^+\ne \emptyset$, Theorem \ref{T5} below establishes the existence of a coexistence EE solution of \eqref{model-eq1} for small diffusion rates of the population under some appropriate hypothesis.

\subsection{Main Results}
We state our main results. Our first result reads as follows.

\begin{tm}\label{T1}
\begin{itemize}
\item[\rm (i)]
There is a positive number $m_0>0$ such that 
\begin{equation}\label{T1-e0}
    \liminf_{t\to\infty}\min_{j\in\Omega}S_j(t)\ge m_0
\end{equation}
for every solution $(S(t),I_1(t),I_2(t))$ of \eqref{model-eq1} with initial data in $\mathcal{E}$.
\item[\rm (ii)]
The DFE  is linearly stable if $\mathcal{R}_0(N)<1$ and unstable if $ \mathcal{R}_0(N)>1$. Furthermore, the following conclusions hold.
\begin{itemize}
    \item[\rm (ii-1)]
If $\mathcal{R}_{0,l}(N)\le \frac{1}{k}$ for some $l=1,2$, then strain-l eventually dies out. In particular if $\mathcal{R}_0(N)\le \frac{1}{k}$, then the DFE is globally stable. 

\item[\rm (ii-2)] If $\mathcal{R}_{0,l}(N)<1$ for some $l=1,2$, then there is $d^l>0$ such that the strain-l eventually dies out for any diffusion of susceptible population $d_S>d^l$.  In particular if $\mathcal{R}_0(N)<1$, then there is $d^*>0$ such that the DFE is globally stable for any diffusion of susceptible population $d_S>d^*$.
\item[\rm (ii-3)]
If $\mathcal{R}_{0,l}(N)>1$ for each $l=1,2$, then   the disease is persistence in the sense that there is a positive number $m_*>0$ such that for every solution $(S(t),I_{1}(t),I_2(t))$ with a positive initial data in $\mathcal{E}$, 
\begin{equation}\label{T1-e1}
    \liminf_{t\to \infty}\min_{j\in\Omega}\sum_{l=1}^2I_{l,j}(t)\ge m_*.
\end{equation}
Furthermore,
\begin{equation}\label{T1-e2}
    \limsup_{t\to\infty}\max_{j\in\Omega}S_j(t)\le s_{\max} \quad \text{and}\quad \liminf_{t\to\infty}\min_{j\in\Omega}S_j(t)\ge s_{\min},
\end{equation}
where $s_{\max}:=\max_{l=1,2}\max_{j\in\Omega}\frac{\gamma_{l,j}}{\beta_{l,j}}$ and $s_{\min}:=\min_{l=1,2}\min_{j\in\Omega}\frac{\gamma_{l,j}}{\beta_{l,j}}$.

\end{itemize}

\end{itemize}
\end{tm}
Theorem \ref{T1}-{\rm (i)} indicates that the susceptible population persists uniformly, hence it will never be driven to extinction. On the other hand, thanks to Theorem \ref{T1}-{\rm (ii-1)}, when $\mathcal{R}_{0}(N)\le \frac{1}{k}$, irrespective of the dispersal rate of the susceptible population, the disease will eventually go extinct. Recall that $\mathcal{R}_0(N)$ is independent of the dispersal rate of the susceptible population, and  $k$ is the total number of patches. In the  particular case of $k=1$, the model reduces the single-patch ODE model for which it is well known that the DFE is globally stable if and only if  $\mathcal{R}_0(N)\le 1.$ So,  our result is sharp in the case of single-patch model. Note  also that,  for large diffusion rate of the susceptible population, Theorem \ref{T1}-{\rm (ii-2)} shows that the DFE is globally stable when $\mathcal{R}_0(N)<1$. When $k\ge 2$, that is there are at least two patches, the question of the global stability of the DFE  remains open when $\frac{1}{k}<\mathcal{R}_0(N)\le 1$ and the dispersal rate of the susceptible is small. We expect that the answer to this question will dependent delicately on the details on the spatial distribution of the  recovery and  transmission rates and how the dispersal rates are selected. For example, when dispersal rates are uniform, Theorem \ref{T4} below shows that the DFE is globally stable whenever $\mathcal{R}_0(N)\le 1$. Theorem \ref{T3} below also confirms that the same conclusion holds when the local reproduction functions are spatially homogeneous. In the case of the single-strain PDE-SIS model \eqref{PDE-model}, Castellano and Salako \cite{CS2023b} recently proved, under some appropriate hypotheses, the existence of at least two EE solutions when $\mathcal{R}_{0,1}(N)<1$ and $d_S$ is sufficiently small. We suspect that such result will also hold for the multiple patches SIS model \eqref{model-eq1}.

When each strain's basic reproduction number is bigger than one, Theorem \ref{T1}-{\rm (ii-3)} establishes the uniform persistence of the disease, but does not exclude the possibility of the extinction of one strain. Hence, it becomes pertinent to examine the conditions on the parameters of the model which lead to a competition-exclusion of the strains. This is carried out in Theorems \ref{T2}, \ref{T3}, and \ref{T4} below. 

Our first result on the competition-exclusion of the strains when one local reproductive number is spatially homogeneous reads a follows.

\begin{tm}\label{T2}
Suppose that $\mathcal{R}_{0,l}(N)>1$, $l=1,2$, and $\mathfrak{R}_{1,j}$ is constant in $j\in\Omega$. 

\begin{itemize}
    \item[\rm (i)] If $ \mathcal{R}_{0,1}(N)>\mathcal{R}_{0,2}(N)$, then the single strain-$1$ EE is linearly stable. Furthermore, if $\Sigma_1=\Omega$, then the  single strain-$1$ EE is globally stable with respect to positive initial data.

    \item[\rm (ii)] If $ \mathcal{R}_{0,1}(N)<\mathcal{R}_{0,2}(N)$, then the single strain-$1$ is unstable. Furthermore, if $\Sigma_2=\Omega$, then  for solutions of \eqref{model-eq1} with positive initial data, the strain-$1$ infected population eventually goes extinct while strain-$2$ persists uniformly in time. 
    
\end{itemize}
    
\end{tm}
Note that Theorem \ref{T2} can also be stated if we interchange the roles of strain-1 and strain-2. Now, suppose that strain-1 local reproduction function is spatially homogeneous, that is, its transmission and recovery rates are proportional to each other. If it also has the largest basic reproduction number, Theorem \ref{T2}-{\rm (i)} indicates that it cannot be invaded at equilibrium by an initially small size of the infected population with strain-2. If in addition, it strictly locally maximizes the local reproduction functions on all patches, then it drives the other strain to extinction. However, if strain-1 has a smaller basic reproduction number, Theorem \ref{T2}-{\rm (ii)} shows that it will be invaded at equilibrium by the other strain. Interestingly, these results hold irrespective of the dispersal rates of the population and suggest that the asymptotic dynamics of the solutions are independent of the dispersal rates when the local reproduction functions are spatially homogeneous. Our next result answers this question with an affirmation and provide a complete dynamics of solutions of \eqref{model-eq1} when both local reproduction functions are spatially homogeneous.

\begin{tm}\label{T3}
    Suppose that $\mathfrak{R}_{l,j}$ is constant in $j\in\Omega$ for each $l\in\{1,2\}$.
    \begin{itemize}
        \item[\rm (i)] If $\mathcal{R}_0(N)\le 1$, then the DFE is globally stable.

        \item[\rm (ii)] If $\mathcal{R}_{0,l}(N)>\max\{1,\mathcal{R}_{0,q}(N)\}$, for $l\ne q\in\{1,2\}$. Then, the strain-$l$ EE solution is globally stable.
    \end{itemize}
\end{tm}

When the local reproduction functions are spatially homogeneous, Theorem \ref{T3} establishes that the strain with the largest basic reproduction number drives the other strain to extinction, irrespective of the dispersal rates of the population. In this case, the dispersal rates of the populations play no role on the asymptotic dynamics of solutions as it is determined by that of the model with no dispersal rate. To gain some understanding on how the dispersal rates may affect the global dynamics of the solutions of \eqref{model-eq1}, in the next two results, we study the large time behaviors of solution in the special case of the following assumption.
\medskip

\noindent{\bf (A2)} $d:=d_1=d_2=d_S$.

\medskip

 \noindent Assumption {\bf (A2)} means that  the population dispersal rate is uniform and independent of the relevant subgroups of the population. In  our next result, we again establish the competition-exclusion of the strains when at least one has its basic reproduction number less or equal to one.

\begin{tm}\label{T4}
    Suppose that hypothesis {\bf (A2)} holds.   The following conclusions hold.
    \begin{itemize}
        \item[\rm (i)] If $\mathcal{R}_0(N)\le 1$, then the DFE is globally stable.
        \item[\rm (ii)] If $\mathcal{R}_{0,l}(N)>1\ge\mathcal{R}_{0,q}(N)$, $l\ne q\in\{1,2\}$, then the strain-l EE is globally stable with respect to positive initial data. 
    \end{itemize}
\end{tm}

  When hypothesis {\bf (A2)} holds, that is the population's dispersal rate is uniform and independent of the relevant subgroups,  the basic reproduction numbers serve as threshold values for the global dynamics of solutions of \eqref{model-eq1}. In this case, it follows from Theorem \ref{T4} that  any strain whose basic reproduction number is less or equal to one eventually dies out. Note that, when both strain's reproduction numbers are bigger than one,  Theorems \ref{T2}-\ref{T3} provide sufficient conditions for the extinction of at least one strain. So, a natural question is to know whether both strains can coexist, and if so, provide some sufficient conditions on the parameters of the model which guarantee this.   To tackle these questions, it seems appropriate to first introduce the strains' {\it invasion numbers}, which measure the ability for a strain to invade another strain when rare.

   Fix $l\ne p=1,2$ and suppose that $\min\{\mathcal{R}_{0,1}(N),\mathcal{R}_{0,2}(N)\}>1$. It follows from \cite[Theorem 4]{LP2023} that \eqref{model-eq1} has at least one strain-$l$ EE solution. Under some further assumptions on the parameters of the model, \cite[Theorem 4]{LP2023} established the uniqueness of the single-strain EE. When these assumptions do not hold, it is not clear whether single-strain EE are unique. Multiplicity of single-strain EE solutions for the PDE-SIS model \eqref{PDE-model} was recently established in \cite{CS2023b}.  So, for convenience, let $$\mathcal{E}_l:=\{(S,I_1,I_2)\in\mathcal{E}\ :\ I_{p}={\bf 0}\}$$
  and  $\mathcal{E}^*_l $ denote the set of single-strain EE solutions of \eqref{model-eq1} in $\mathcal{E}_l$. It is clear that both $\mathcal{E}_l$ and $\mathcal{E}^*_l$ are compact sets and invariant under the semiflow of solutions of \eqref{model-eq1}. Moreover, we have that ${\bf E}^0\notin \mathcal{E}^*_l$ and $\mathcal{E}^*_1\cap \mathcal{E}^*_2=\emptyset$. Given ${\bf E}_1^*=(S_1^*,I_1^*,{\bf 0})\in \mathcal{E}^*_1$ and ${\bf E}_2^*=(S_1^*,{\bf 0},I_2^*)\in \mathcal{E}^*_1$ define 
  \begin{equation}\label{GG1}
      \tilde{\mathcal{R}}_2({\bf E}_1^*)=\rho\Big({\rm diag}(\beta_2\circ S_1^*)\mathcal{V}_2^{-1}\Big),\quad\quad  \tilde{\mathcal{R}}_1({\bf E}_2^*)=\rho\Big({\rm diag}(\beta_1\circ S_2^*)\mathcal{V}_1^{-1}\Big),
  \end{equation}
  \begin{equation}\label{GG2}
      \tilde{\mathcal{R}}_1(N)=\min\{\tilde{\mathcal{R}}_1({\bf E}^*_2) :\ {\bf E}^*_2\in \mathcal{E}^*_2\}\quad \text{and}\quad \tilde{\mathcal{R}}_2(N)=\min\{\tilde{\mathcal{R}}_2({\bf E}^*_1) :\ {\bf E}^*_1\in \mathcal{E}^*_1\}.
  \end{equation}
  The quantity $\tilde{\mathcal{R}}_l(N)$ is the strain-l's invasion number. Note that when {\bf (A2)} holds, it follows from \cite[Theorem 9]{LP2023} that $\mathcal{E}^*_l$ consists of a single element for each $l=1,2$. For any element $ E\in \mathcal{E}$ and a subset $\mathcal{O}\subset \mathcal{E}$, let 
  $$
{\rm dist}(E,\mathcal{O}):=\inf\{\|E-P\|_{\infty}:\ P\in\mathcal{O}\}
  $$
  denote the distance from the point $E$ to the set $\mathcal{O}$.  Our result on the coexistence EE of \eqref{model-eq1} reads as follows.
   \begin{tm} \label{T5-1} Suppose that 
   $\min\{\mathcal{R}_{0,1}(N),\mathcal{R}_{0,2}(N)\}>1$. Let $\tilde{\mathcal{R}}_{1}(N)$  and $\tilde{\mathcal{R}}_{2}(N)$  be defined as in \eqref{GG2}. 
   \begin{itemize}
       \item[\rm (i)] If $\tilde{\mathcal{R}}_{1}(N)>1$, then 
       \begin{equation}\label{2-GG1}
           \liminf_{t\to\infty}\frac{1}{t}\int_{0}^t{\rm dist}((S(\tau),I_1(\tau),I_2(\tau))(t),\mathcal{E}^*_2)d\tau\ge \frac{\beta_{\min}}{\beta_{\max}}\frac{(\tilde{\mathcal{R}}_1(N)-1)}{\tilde{\mathcal{R}}_1(N)}
       \end{equation}
       for every solution $(S,I_1,I_2)(t)$ of \eqref{model-eq1} with initial data satisfying $\|I_1(0)\|_{\infty}>0$. Furthermore, if in addition $\mathcal{E}_2^*\cup\{{\bf E}^0\}$ is the global attractor for classical solutions of \eqref{model-eq1} with initial data in $\mathcal{E}_2$, then there is a positive  number $\sigma_1^*>0$ such that 
       \begin{equation}\label{2-GG3}
           \liminf_{t\to\infty}\min_{j\in\Omega}I_{1,j}(t)\ge \sigma_1^*,
       \end{equation} 
       for every solution $(S,I_1,I_2)(t)$ of \eqref{model-eq1} with initial data satisfying $\|I_1(0)\|_{\infty}>0$.
       Similar result holds for the strain-2.
       \item[\rm (ii)] Suppose that $\min\{\tilde{\mathcal{R}}_{1}(N),\tilde{\mathcal{R}}_{2}(N)\}>1 $. If $\mathcal{E}_1^*\cup\{{\bf E}^0\}$ and $\mathcal{E}_2^*\cup\{{\bf E}^0\}$   are the global attractor for classical solutions of \eqref{model-eq1} with initial data $\mathcal{E}_1$ and $\mathcal{E}_2$, respectively,  then \eqref{model-eq1} has at least one coexistence EE solution.
   \end{itemize}
  \end{tm}

 Suppose that $\mathcal{R}_{0,l}(N)>1$ for each $l=1,2$ and hypothesis {\bf (A2)} holds. Let  $d>0$ be the uniform dispersal rate as in {\bf (A2)}.  It follows from \cite[Theorem 3]{LP2023} that \eqref{model-eq1} has a unique single-strain EE solutions   ${\bf E}_1^*:=\Big(\frac{N}{k}{\bf 1}-I_1^*,I_1^*,{\bf 0}\Big)$ and  ${\bf E}_2^*:=\Big(\frac{N}{k}{\bf 1}-I_2^*,{\bf 0},I_2^*\Big)$, where $I_1^*$ is the unique positive solution of the multiple-patch logistic equation
 \begin{equation}\label{I-star-eq}
     0=d\mathcal{L}I_1^*+\Big(\beta_l\circ\Big(\frac{N}{k}{\bf 1}-I_1^*\Big)-\gamma_l\Big)\circ I_l^*.
 \end{equation}
Furthermore, it holds that
\begin{equation}\label{Invasion-number}
    \tilde{\mathcal{R}}_l(N)=\rho\Big(\Big(\frac{N}{k}\mathcal{F}_l-{\rm diag}(\beta_l\circ I_p^*)\Big)\mathcal{V}_l^{-1}\Big)\quad p\ne l\in\{1,2\}.
\end{equation}
As an application of Theorem \ref{T5-1}, we can state the following result.

\begin{tm}\label{T5}  Suppose that hypothesis {\bf (A2)} holds and $\min\{\mathcal{R}_{0,1}(N),{\mathcal{R}}_{0,2}(N)\}>1$. Let ${\bf E}_1^*$ and ${\bf E}_2^*$ denote the single-strain EE solutions of \eqref{model-eq1}. Fix $l\ne p\in\{1,2\}$.

\begin{itemize}
    \item[\rm (i)] If $\tilde{\mathcal{R}}_{l}(N)>1$, the single-strain EE, ${\bf E}_p^*$, is unstable and strain-l is uniformly persistent in the sense of \eqref{2-GG3}.
    
    \item[\rm (ii)] If $\min\{\tilde{\mathcal{R}}_{1}(N),\tilde{\mathcal{R}}_{2}(N)\}>1 $, then \eqref{model-eq1} has at least one coexistence EE solution.
    \item[\rm (iii)] If $\Sigma_1\cap H_1^+$ and $\Sigma_2\cap H^+_2$ are both nonempty sets, then there is $d^*>0$ such that for any $0<d<d^*$,  $\min\{\tilde{\mathcal{R}}_{1}(N),\tilde{\mathcal{R}}_{2}(N)\}>1 $. 
\end{itemize}
    
\end{tm}

When hypothesis {\bf (A2)} holds, $\Sigma_1\cap H_1^{+}\ne \emptyset$ and $\Sigma_2\cap H_2^+\ne \emptyset$, Theorem \ref{T5} establishes the existence of a coexistence EE for small dispersal rate of the population. We expect that the same conclusion should hold when hypothesis {\bf (A2)} is dropped. The questions of uniqueness and  global stability of coexistence EE will be explored in our future work. Another interesting question is to investigate how the spatial distribution of the population at coexistence EE solutions depend on the dispersal rates. This question will also be investigated in our future work.

\subsection{Numerical Simulations and Discussion}

To further understand the dynamics of solutions of \eqref{model-eq1}, we perform some simulations. Some of the simulations illustrate our theoretical results, while others explore some aspects that our theoretical results did not cover. For all the simulations, we consider two patches models, that is $k=2$, $\Omega=\{1,2\}$ and take $L_{1,1}=L_{2,2}=0$ and $L_{1,2}=L_{2,1}=1$.

\medskip

\begin{Sim}\label{Simulation1}
We  fix the parameters: $\beta_{1}=(2,3)^T$, $\beta_2=(1,4)^T$, $\gamma_1=(1,2)^T$, $\gamma_{2}=(2,3)^T$, $d_S=3$, $d_1=1$, and $d_2=2$. Next, we perform three simulations with three choices of the total population $N$. {\rm (a)} In Fig \ref{fig1-3}, we take $S(0)=(0.05,0.05)^T$, $I_{1}(0)=(0.05,0.05)^T$, and $I_2(0)=(0.05,0.05)^T$.   Hence, $N=0.3$,  $\mathcal{R}_{0,1}(N)=0.1627$, $\mathcal{R}_{0,2}(N)=0.2535$, and $\mathcal{R}_0(N)<1$. The simulations indicate an extinction of the disease.   {\rm (b)} In Fig \ref{fig1-2}, we take $S(0)=(0.25,0.25)^T$, $I_{1}(0)=(0.25,0.25)^T$, and $I_2(0)=(0.25,0.25)^T$.   Hence, $N=1.5$,  $\mathcal{R}_{0,1}(N)=1.2674$, $\mathcal{R}_{0,2}(N)=0.81375$ and $\mathcal{R}_0(N)=\mathcal{R}_{0,1}(N)>1>\mathcal{R}_{0,2}(N)$. We also observe an exclusion of strain-2 while strain-1 persists. {\rm (c)} In Fig\ref{fig1-1}, we take $S(0)=(1,2)$, $I_{1}(0)=(0.5,0.5)$, and $I_2(0)=(0.5,0.5)$. Hence, $N=4$,  $\mathcal{R}_{0,1}(N)=2.7125$, $\mathcal{R}_{0,2}(N)=4.2247$ and $\mathcal{R}_0(N)=\mathcal{R}_{0,1}(N)>1$. In this case, we observe from Fig \ref{fig1-1} that strain-1 persists while strain-2 eventually dies out. This agrees with persistence of at least one strain of the disease as predicted by Theorem \ref{T1}-{\rm (ii-3)}. However, we notice an exclusion of strain-2 even though its basic reproduction number is bigger than one. 

\medskip

\begin{figure}[ht]
     \centering

         \begin{subfigure}[b]{0.3\textwidth}
         \centering
         \includegraphics[width=\textwidth]{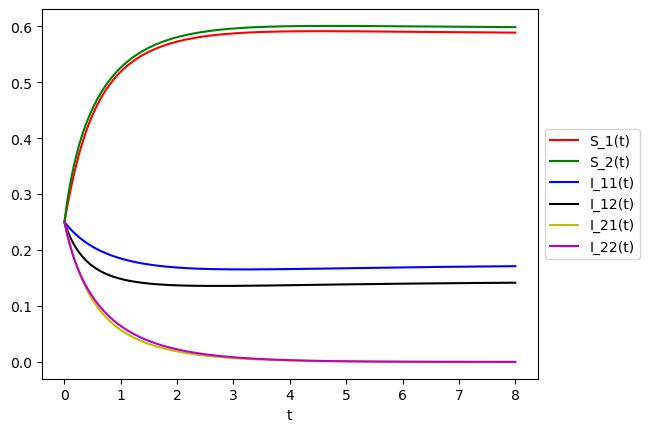}
         \caption{}
         \label{fig1-3}
     \end{subfigure}
      \hfill
     \begin{subfigure}[b]{0.3\textwidth}
         \centering
         \includegraphics[width=\textwidth]{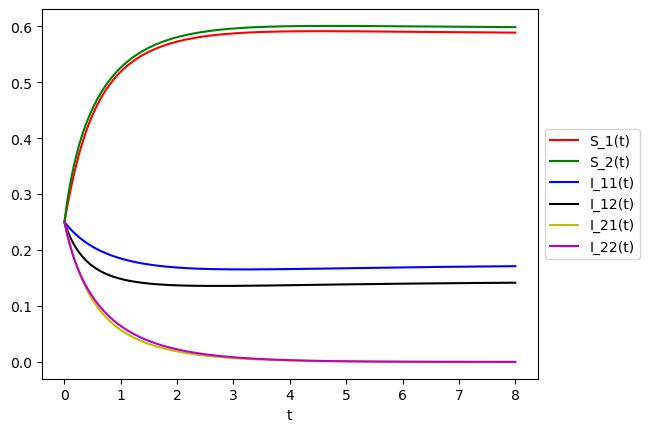}
         \caption{}
         \label{fig1-2}
     \end{subfigure} 
      \hfill
     \begin{subfigure}[b]{0.3\textwidth}
         \centering
         \includegraphics[width=\textwidth]{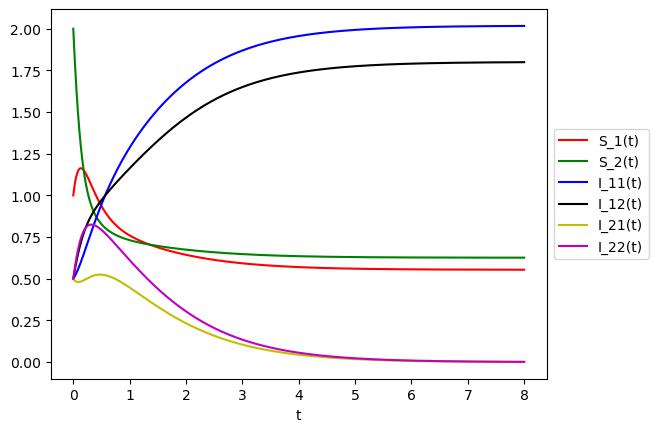}
         \caption{}
         \label{fig1-1}
     \end{subfigure}

        \caption{Numerical simulations illustrating disease extinction when $\mathcal{R}_0(N)<1$, and competition-exclusion of the strains when $\mathcal{R}_0(N)>1$.}
        \label{fig1}
\end{figure}
    
\end{Sim}

\begin{Sim}\label{Simulation2} We fix the parameters: $\beta_1=(4,6)^T$, $\beta_{2}=(1,4)^T$, $\gamma_1=(2,3)^T$, and $\gamma_2=(2,3)^T$. So that $\mathfrak{R}_1=(2,2)^T$ is constant and $\|\mathfrak{R}_2\|_{\infty}<2$.  We also fix the initial data to $S_0=(1,2)$, $I_{1}=(0.5,0.5)$ and $I_{2}=(0.5,0.5)$. We then perform three simulations for different choices of the migration rates: (a) $d_S=3$, $d_1=1$ and $d_2=2$ in Figure \ref{fig2-1}; (b) $d_S=4$, $d_1=6$ and $d_2=2$ in Figure \ref{fig2-2}; and $d_S=10$, $d_1=0.5$ and $d_2=20$ in Figure \ref{fig2-3}. We observe that strain-2 eventually dies out in all the three simulations irrespective of the choices of the migration rates. These agree with the conclusion of Theorem \ref{T2}-{\rm (i)}.

\begin{figure}[h]
     \centering
     \begin{subfigure}[b]{0.3\textwidth}
         \centering
         \includegraphics[width=\textwidth]{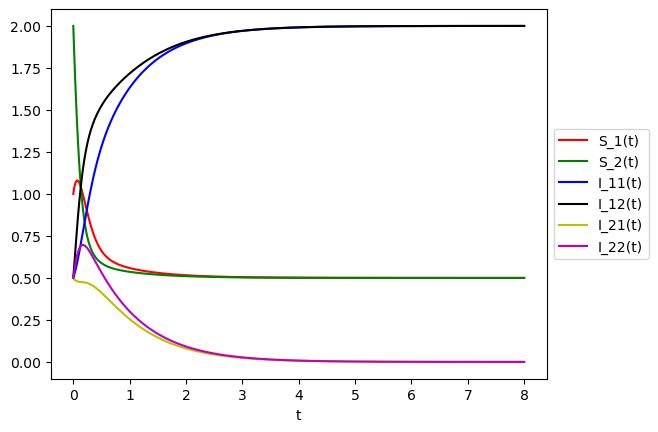}
         \caption{}
         \label{fig2-1}
     \end{subfigure}
     \hfill
     \begin{subfigure}[b]{0.3\textwidth}
         \centering
         \includegraphics[width=\textwidth]{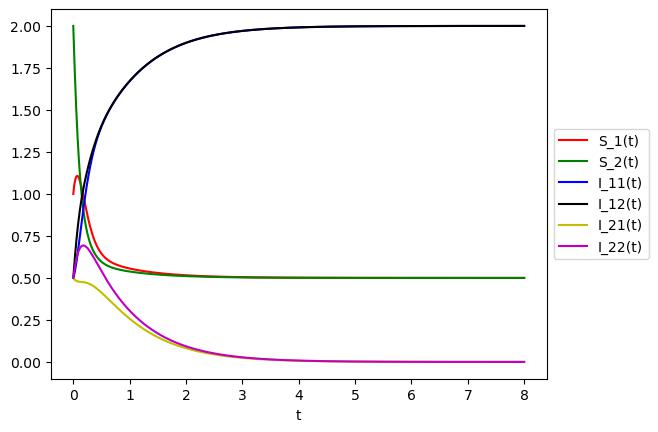}
         \caption{}
         \label{fig2-2}
     \end{subfigure}
     \hfill
     \begin{subfigure}[b]{0.3\textwidth}
         \centering
         \includegraphics[width=\textwidth]{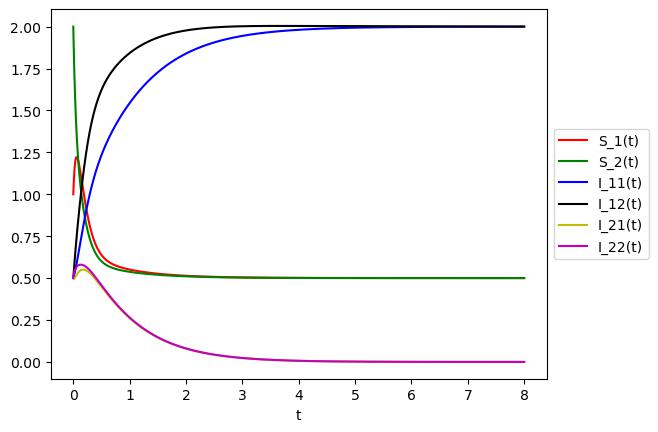}
         \caption{}
         \label{fig2-3}
     \end{subfigure}
        \caption{Numerical simulations illustrating extinction of strain-2 when strain-1 local reproductive function is constant and locally strictly maximizes the local reproductive functions on all patches.}
        \label{fig2}
\end{figure}
    
\end{Sim}

    \begin{Sim}\label{Simulation3}  We fix the parameters: $\beta_1=(\frac{2}{3},1)^T$, $\beta_{2}=(1,4)^T$, $\gamma_1=(2,3)^T$, and $\gamma_2=(2,3)^T$. So that $\mathfrak{R}_1=(\frac{1}{3},\frac{1}{3})^T$ is constant and $\mathfrak{R}_{2,\min}>\frac{1}{3}$.  We also fix the initial data to $S_0=(1,2)$, $I_{1}=(1,1)$ and $I_{2}=(1,1)$. We then perform three simulations for different choices of the migration rates: (a) $d_S=5$, $d_1=1$ and $d_2=2$ in Figure \ref{fig3-1}; (b) $d_S=0.005$, $d_1=0.5$ and $d_2=2$ in Figure \ref{fig3-2}; and $d_S=10$, $d_1=0.005$ and $d_2=2$ in Figure \ref{fig3-3}. We observe that strain-1 eventually dies out in all the three simulations irrespective of the choices of the migration rates. These agree with the conclusion of Theorem \ref{T2}-{\rm (ii)}.

\medskip
  
\begin{figure}[h!]
     \centering
     \begin{subfigure}[b]{0.3\textwidth}
         \centering
         \includegraphics[width=\textwidth]{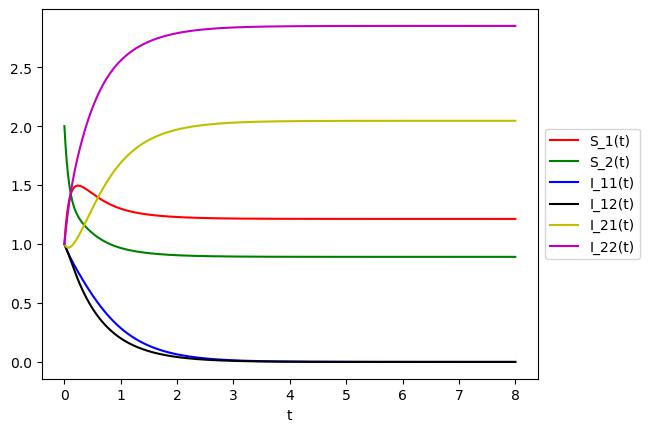}
         \caption{}
         \label{fig3-1}
     \end{subfigure}
     \hfill
     \begin{subfigure}[b]{0.3\textwidth}
         \centering
         \includegraphics[width=\textwidth]{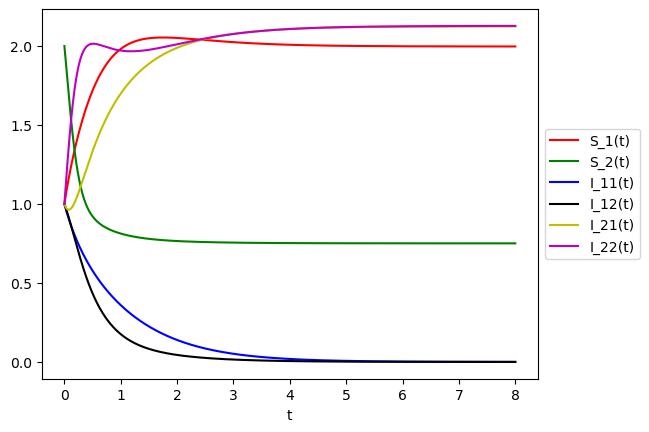}
         \caption{}
         \label{fig3-2}
     \end{subfigure}
     \hfill
     \begin{subfigure}[b]{0.3\textwidth}
         \centering
         \includegraphics[width=\textwidth]{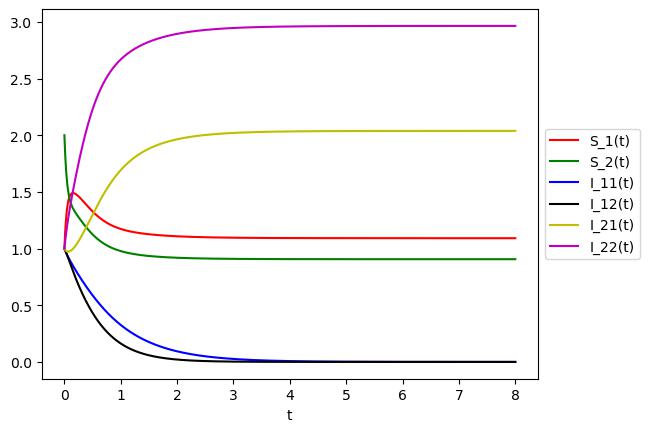}
         \caption{}
         \label{fig3-3}
     \end{subfigure}
        \caption{Numerical simulations illustrating extinction of strain-1 when its local reproductive function is constant and locally strictly minimizes the local reproductive functions on all patches.}
        \label{fig3}
\end{figure}
    
\end{Sim}

\begin{Sim}\label{Simlation4}  We fix the parameters: $\beta_1=(2,3)^T$, $\beta_{2}=(1,4)^T$, $\gamma_1=(2,3)^T$, and $\gamma_2=(2,3)^T$. So that $\mathfrak{R}_1=(1,1)^T$ is constant and $\mathfrak{R}_{2,\min}<1<\mathfrak{R}_{2,\max}$. Next, we fix the diffusion rate $d_2=2$ of the infected population with strain-2.   Finally, we fix the initial data to $S_0=(1,2)$, $I_{1}=(2,1)$ and $I_{2}=(4,1)$. We then perform three simulations for different choices of the migration rates $d_S$ and $d_1$:  $d_S=5$ and $d_1=1$ in Figure \ref{fig4-1};  $d_S=35$ and $d_1=35$  in Figure \ref{fig4-2}; and $d_S=40$ and $d_1=40$ in Figure \ref{fig4-3}. We notice that both strains coexist when the migration rates are small as illustrated by Figure \ref{fig4-1}, while strain-1 dies out for large diffusion rates as seen from Figures \ref{fig4-2} and \ref{fig4-3}. These simulations complement Theorem \ref{T2} and suggest that the dynamics of the disease delicately depends on the choices of the migration rates when one of the strain local reproductive function is constant and neither maximizes nor minimizes the local reproductive functions.

\medskip

\begin{figure}[h!]
     \centering
     \begin{subfigure}[b]{0.3\textwidth}
         \centering
         \includegraphics[width=\textwidth]{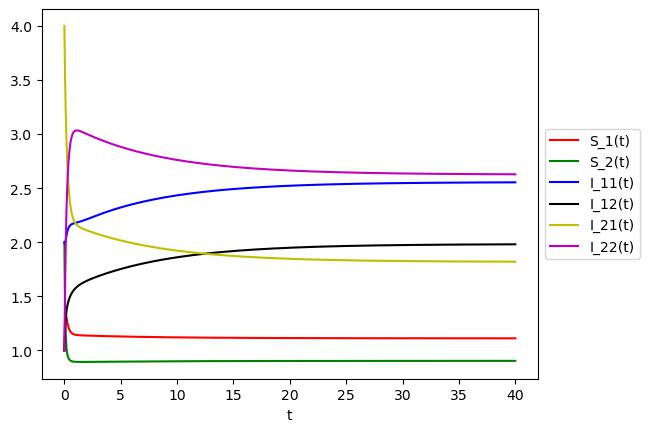}
         \caption{}
         \label{fig4-1}
     \end{subfigure}
     \hfill
     \begin{subfigure}[b]{0.3\textwidth}
         \centering
         \includegraphics[width=\textwidth]{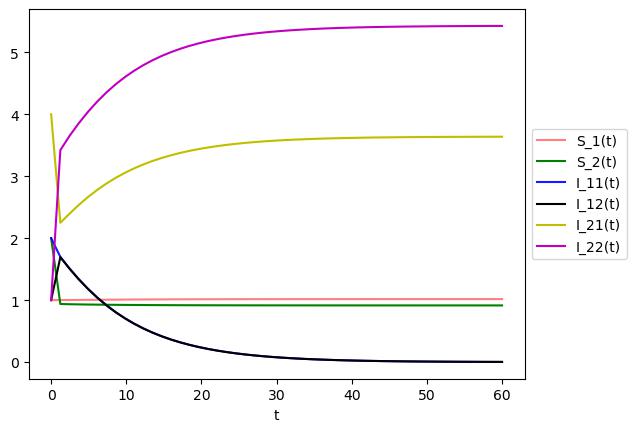}
         \caption{}
         \label{fig4-2}
     \end{subfigure}
     \hfill
     \begin{subfigure}[b]{0.3\textwidth}
         \centering
         \includegraphics[width=\textwidth]{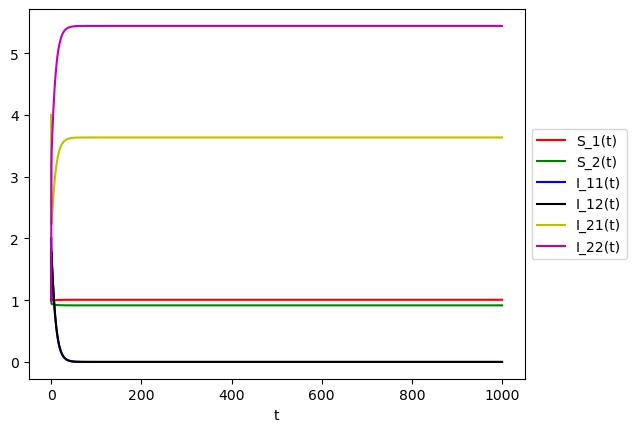}
         \caption{}
         \label{fig4-3}
     \end{subfigure}
        \caption{Numerical simulations illustrating both competition-exclusion and coexistence of the two strains when strain-1 local reproductive function is constant and  locally strictly maximizes the local reproductive functions on exactly one patch.}
        \label{fig4}
\end{figure}
    
\end{Sim}

\begin{Sim}\label{Simulation5}

We fix the parameters: $\beta_1=(2,3)^T$, $\beta_{2}=(1,4)^T$, $\gamma_1=(2,3)^T$, and $\gamma_2=(2,3)^T$. So that $\mathfrak{R}_1=(1,1)^T$ is constant and $\mathfrak{R}_{2,\min}<1<\mathfrak{R}_{2,\max}$.   We also fix the initial data to $S_0=(1,2)$, $I_{1}=(2,1)$ and $I_{2}=(4,1)$, so that $N=11$. Hence, $\Sigma_1\cap H_1^+=\{1\}$ and $\Sigma_2\cap H^+=\{2\}$  We then perform three simulations for  different choices of equal migration rates $d:=d_S=d_1=d_2$: (a) $d=0.005$ in Figure \ref{fig5-1}; (b) $d= 35 $  in Figure \ref{fig4-2}; and $d=40$ in Figure \ref{fig5-3}. All the three simulations illustrate the coexistence of the strain. Hence, support the conclusions of Theorem \ref{T5}.  When the diffusion rates are sufficiently small, Fig \ref{fig5-1} indicates a spatial segregation of the infected populations. However, for large diffusion rates, Figures \ref{fig5-2} and \ref{fig5-3} show that both strains persists uniformly on all patches.

\medskip

\begin{figure}[h!]
     \centering
     \begin{subfigure}[b]{0.3\textwidth}
         \centering
         \includegraphics[width=\textwidth]{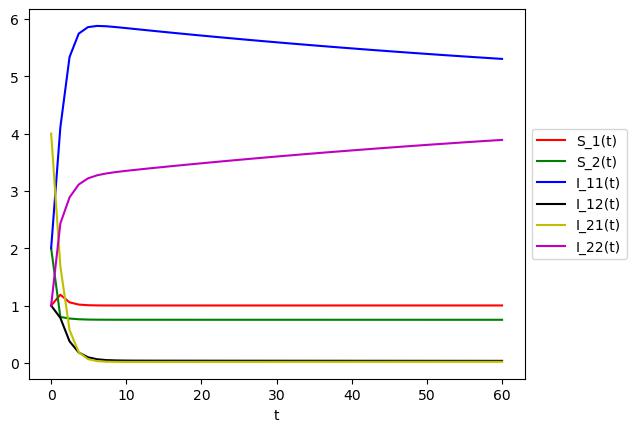}
         \caption{}
         \label{fig5-1}
     \end{subfigure}
     \hfill
     \begin{subfigure}[b]{0.3\textwidth}
         \centering
         \includegraphics[width=\textwidth]{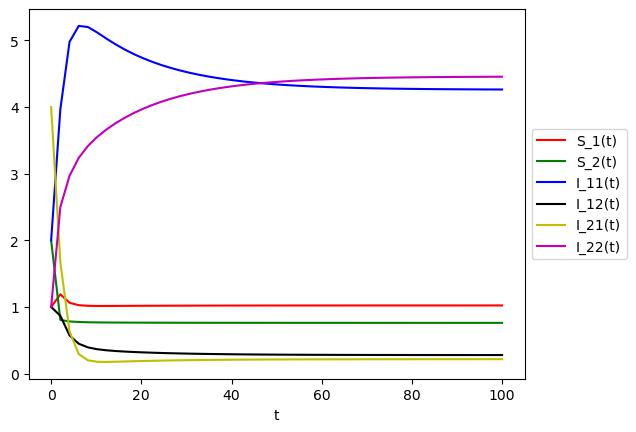}
         \caption{}
         \label{fig5-2}
     \end{subfigure}
     \hfill
     \begin{subfigure}[b]{0.3\textwidth}
         \centering
         \includegraphics[width=\textwidth]{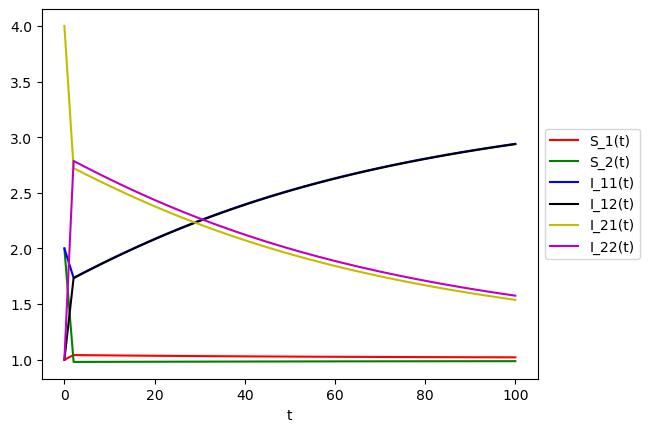}
         \caption{}
         \label{fig5-3}
     \end{subfigure}
        \caption{Numerical simulations illustrating coexistence both strains for equal dispersal rates of all the subgroup of the population.}
        \label{fig5}
\end{figure}
    
\end{Sim}

\subsubsection{Discussion}
We examined the dynamics of solutions of a two-strain epidemic model in patchy environment. To this end, we introduced the basic reproduction number $\mathcal{R}_0(N)$, and then first discussed the extinction of the disease under some sufficient conditions on the parameters of the models. In particular, Theorem \ref{T1} suggests that the disease would be eventually eradicated if either the basic reproduction number is smaller than the reciprocal of the total number of patches, or the migration rate of the susceptible population is sufficiently large and $\mathcal{R}_0(N)<1$. However, at least one strain of the disease would persist if $\mathcal{R}_0(N)>1$.  Observing that the latter assertion does not exclude the possibility of extinction of at least one strain, we then investigated sufficient conditions which lead to the competition-exclusion principle. Simulations in Figure \ref{fig1} provide an illustration of the dynamics of the disease as suggested by Theorem \ref{T1}.

As noted in the introduction, in the celebrated work \cite{BT1989},  Bremermann and Thieme showed that, when diffusion rates are neglected and there is only one patch, any strain that does not maximize the reproduction number would eventually die out. In Theorems \ref{T2} and \ref{T3}, we showed that this conclusion extends to the multiple patches model with diffusion rates if the local reproduction function is patch independent. In fact, the conclusion of Theorem \ref{T2} is somehow stronger as it only requires that one strain to have its local reproductive function to be constant and  to either maximize or minimize the local reproductive function on all patches. These results were perfectly illustrated by our simulations \ref{fig2}, \ref{fig3}, and \ref{fig4}. Interestingly, our simulations in Figure \ref{fig4} indicate that the conclusions of Theorem \ref{T2} might not hold if some of the assumptions are dropped. 

We further investigated some sufficient conditions on the parameters of the model which may lead to the coexistence of the strains. In this direction, we introduced the invasion numbers and then established the existence of coexistence EE when these numbers are bigger than one in Theorems \ref{T5-1} and \ref{T5}. Our numerical simulations in Figure \ref{fig5} support these theoretical results. An important question left open by our work is the uniqueness and global stability of the coexistence EE solution. Also, it would be interesting to study the effect of the diffusion rates on the spatial distributions of the coexistence EE as such information might help to developing and implementing adequate and effective disease control strategies. Another aspect that is not considered by our work is the effect of a total lockdown of one part of the population (say for example the movements of the infected populations  are completely restricted) on the dynamics of the disease.   We plan to devote some of our future studies on these open problems.

\section{Preliminaries}\label{Sec3}

We introduce some notations and collect a few  preliminary results in the current section. For convenience, set
  $$\sigma_{\max}=\max_{j\in\Omega, l=1,2}\sigma_{l,j}\quad \text{and}\quad \sigma_{\min}=\min_{j\in\Omega, l=1,2}\sigma_{l,j}\quad \forall\ \sigma\in\{\beta,\gamma\},
$$
\begin{equation*}
    r_{l,j}=\frac{\gamma_{l,j}}{\beta_{l,j}} \quad  l=1,2\ \text{and}\  j\in\Omega,\quad r_{l}=(r_{l,1},\cdots,r_{l,k})^T,\quad \end{equation*}
    \begin{equation*} 
  r_{l,\min}:=\min_{j\in\Omega}r_{l,j},  \quad \text{and}\quad r_{l,\max}:=\max_{j\in\Omega}r_{l,j}\quad \forall\ l=1,2.
\end{equation*}
Note that $\gamma_l=\beta_l\circ r_{l} $ for every $l=1,2$. In general, given a column vector $V\in \mathbb{R}^k$,  set  
$$\|V\|_{\infty}=\max_{j=1,\cdots,k}|V_j|,\quad \|V\|=\sqrt{\sum_{j=1}^kV_j^2},\quad \widehat{V}_j=V_j-\frac{\sum_{i=1}V_i}{k}\quad \text{and}\quad  \widehat{V}=(\widehat{V}_1,\cdots,\widehat{V}_k)^{T}.$$
Given $V$ and $\tilde{V}\in\mathbb{R}^k$, we say that 
$$
V\le_1 \tilde{V}\quad  \text{if}\quad V_{j}\le \tilde{V}_{j}\quad \forall\ j=1,\cdots,k, 
$$
$$
V<_1\tilde{V}\quad \text{if}\quad V\le \tilde{V}\quad \text{and}\quad V_{j}<\tilde{V}_j\ \text{for some}\ j\in\{1,\cdots,k\},
$$
and 
$$ 
V\ll_1\tilde{V}\quad \text{if}\quad V_{j}<\tilde{V}_{j}\quad \forall\ j=1,\cdots,k.
$$
\noindent Next, since $\lambda_*(\mathcal{L})$ is simple,  and $\mathcal{L}$ is real symmetric,  there is an orthogonal matrix $Q$ such that 
$$
Q^T\mathcal{L}Q={\rm diag}(\tilde{d}_1,\cdots,\tilde{d}_k)\quad \text{and}\quad Q^{T}Q=I,
$$
where $\tilde{d}_i$,  $i=1,\cdots,k$ are the eigenvalues of the symmetric matrix $\mathcal{L}$ with $\tilde{d}_1=\lambda_*(\mathcal{L})=0>\max_{i=2,\cdots,k}\tilde{d}_i$. Hence, 
\begin{equation}\label{Z1}
    e^{t\mathcal{L}}=Qe^{t{\rm diag}(\tilde{d}_1,\cdots,\tilde{d}_k)}Q^{T}\quad \forall\ t\ge 0.
\end{equation}
Moreover, if we set $Q_{j}$, $j=1,\cdots,k$, as the column vectors of the square matrix $Q$,  then $\{Q_j\}_{j=1}^k$ is an orthonormal basis of $\mathbb{R}^k$. It then follows from \eqref{Z1} that 
\begin{equation}\label{Z2}
    e^{t\mathcal{L}}U=\sum_{j=1}^ke^{t\tilde{d}_j}<U,Q_j>Q_j\quad \forall\ t\ge 0,\ U\in\mathbb{R}^k,
\end{equation}
where $<,>$ denotes the inner product on $\mathbb{R}^k$. As a result, it follows from \eqref{Z2} and the fact that $\max_{j=2,\cdots,k}\{\tilde{d}_j\}<\tilde{d}_1=0 $ that 
\begin{equation}\label{Z3}
    \Big\|e^{t\mathcal{L}}U-<U,Q_1>Q_1\Big\|\le e^{t\tilde{d}_*}\|U\|\quad \forall\ U\in \mathbb{R}^k
\end{equation}
where $ \tilde{d}_*:=\min_{j=2,\cdots,k}\tilde{d_{j}}<0$. Note that $Q_1=\frac{1}{\sqrt{k}}{\bf 1}$ so that
\begin{equation}\label{Z4}
<U,Q_1>Q_1=\frac{\sum_{j=1}^kU_j}{k}{\bf 1},\quad \text{and}\quad \widehat{U}=U-\frac{\sum_{j=1}^kU_j}{k}{\bf 1}\quad  \forall\ U\in \mathbb{R}^k.
\end{equation}
Observe also that
\begin{equation} \label{Z8}
\|\widehat{U}\|_{\infty}\le 2\|U\|_{\infty},\quad \|\widehat{U}\|\le 2\|U\| \quad \text{and}\quad \|U\circ V\|\le\|V\|_{\infty}\|U\|\quad \forall\ U,V\in\mathbb{R}^k.
\end{equation}
Let $\{B_1,\cdots,B_k\}$ denote the standard orthonormal basis of $\mathbb{R}^k$.  Recall that  since $\mathcal{L}$ is irreducible and cooperative, $\{e^{t\mathcal{L}}\}_{t\ge 0}$ generates a strongly monotone semiflow on $\mathbb{R}^k$. Then 
\begin{equation}\label{Z8-2-2}
 {\bf 0}\ll_1   e^{t\mathcal{L}}B_j \quad \forall\ j=1,\cdots,k,\quad \text{and}\quad  t>0.
\end{equation}
 As a direct consequence of  \eqref{Z8-2-2}, we have the following Harnack's inequality type, which plays an important in the proofs of our main results.
\begin{lem}\label{Harnck-lemma} Let $d>0$ and $M\in C(\mathbb{R}_+:[\mathbb{R}]^k)$ such that 
\begin{equation*}
    m_{\infty}:=\sup_{t\ge 0}\|M(t)\|_{\infty}<\infty.
\end{equation*}
 Let $\tilde{\mathcal{L}}$ be a $k\times k$ irreducible square matrix  generating a strongly positive semigroup $\{e^{t\tilde{\mathcal{L}}}\}_{t\ge 0}$ on $\mathbb{R}^k$.  
Then there is a positive number $\tilde{c}_{d,m_{\infty}}$, such that any nonnegative solution $U(t)$ of 
\begin{equation}\label{Harnack-eq1}
    \frac{dU}{dt}=d\tilde{\mathcal{L}}U +M(t)\circ U, \ t>0,
\end{equation}
satisfies
\begin{equation}\label{Harnack-eq2}
    \|U(t)\|_{\infty}\le \tilde{c}_{d,m_{\infty}}U_{\min}(t)\quad \forall\ t\ge 1.
\end{equation}
    
\end{lem}
\begin{proof} Without loss of generality, we may suppose $\lambda_*(\tilde{\mathcal{L}})=0$, otherwise we can replace $\tilde{\mathcal{L}}$ by $\tilde{\mathcal{L}}-\lambda_*(\tilde{\mathcal{L}}){\rm diag}({\bf 1})$ and $M(t)$ by $M(t)+d\lambda_*(\tilde{\mathcal{L}}){\rm diag}({\bf 1})$.  Let $\tilde{E}\in\mathbb{R}^k$ be the positive eigenvalue of $\tilde{L}$ satisfying $\min_{j=1,\cdots,k}\tilde{E}_j=1$. Let $U(t)$ be a nonnegative solution of \eqref{Harnack-eq1}. Then 
$$
\frac{dU}{dt}\le d\tilde{\mathcal{L}}U+m_{\infty}U\quad \forall\ t>0.
$$
This implies that 
\begin{equation*}
    U(t+\tau)\le_1 e^{t m_{\infty}}\|U(\tau)\|_{\infty}\tilde{ E}\quad \forall\ t>0,\ \tau>0.
\end{equation*}
Hence
\begin{equation}\label{Harnack-eq3}
   \|U(t+1)\|_{\infty}\le e^{m_{\infty}}\|\tilde{E}\|_{\infty}\|U(t)\|_{\infty}\quad \forall\ t>0. 
\end{equation}
Next, for every $\tau>0$, there is some $j_{\tau}=1,\cdots,k$ such that 
$$ 
\|U(\tau)\|_{\infty} B_{j_{\tau}}\le_1 U(\tau),
$$
where $\{B_1,\cdots,B_k\}$ is the standard orthonormal basis of $\mathbb{R}^k$. Observing that
$$
 d\tilde{\mathcal{L}}(e^{tm_{\infty}}U(\tau+t)) \le_1 \frac{d(e^{tm_{\infty}}U(t+\tau))}{dt}\quad \forall t>0,\ \tau>0,
$$
then 
\begin{equation*}
\|U(\tau)\|_{\infty}e^{td\tilde{\mathcal{L}}}B_{j_{\tau}}\le_1    e^{t m_{\infty}}U(t+\tau)\quad \forall\ t>0,\ \tau >0.
\end{equation*}
In particular, 
\begin{equation}\label{Harnack-eq4}
    e^{-m_{\infty}}\|U(t)\|_{\infty}e^{d\tilde{\mathcal{L}}}B_{j_t}\le U(t+1)\quad \forall\ t>0.
\end{equation}
But, since $\{e^{t\tilde{\mathcal{L}}}\}_{t>0}$ is strictly positive, setting $\tilde{B}^{d}_{l}=e^{d\tilde{\mathcal{L}}}B_{l}$, for each $l=1,\cdots,k$,  we have that 
$$
{\bf 0}\ll_1 \tilde{B}^{*,d}:=(\min_{j\in\Omega}\tilde{B}^d_{j,1},\cdots,\min_{j\in\Omega}\tilde{B}^d_{j,k})^T.
$$
Therefore, in view of \eqref{Harnack-eq4}, we have that 
\begin{equation*}
     e^{-m_{\infty}}\|U(t)\|_{\infty}\tilde{B}^{*,d}\le_1 U(t+1)\quad \forall\ t>0,
\end{equation*}
which  implies that
\begin{equation}\label{Harnack-eq5}
    \Big(\frac{\min_{j\in\Omega}\tilde{B}^{*,d}_j}{e^{m_{\infty}}}\Big)\|U(t)\|_{\infty}\le U_{\min}(t+1)\quad \forall\ t>0.
\end{equation}
Finally, by \eqref{Harnack-eq3} and \eqref{Harnack-eq5}, we have that 
$$ 
\|U(t+1)\|_{\infty}\le \frac{e^{2m}\|\tilde{E}\|_{\infty}}{\min_{j\in\Omega}\tilde{B}^{*,d}}U_{\min}(t+1)\quad \forall\ t>0.
$$
Therefore, \eqref{Harnack-eq2} holds with $\tilde{c}_{d,m_{\infty}}=\frac{e^{2m_{\infty}}\|\tilde{E}\|_{\infty}}{\min_{j\in\Omega}\tilde{B}^{*,d}}>0$.    
\end{proof}
\begin{rk} We note that in Lemma \ref{Harnck-lemma}, it is not required that the irreducible matrix $\tilde{\mathcal{L}}$ to be irreducible. Lemma \ref{Harnck-lemma} generalizes the Hanack's inequality for continuous reaction-diffusion models subject to the homogeneous boundary condtions to the patch model.
    
\end{rk}
Given a solution $(S(t),I_1(t),I_2(t))$ of \eqref{model-eq1} with initial data in $\mathcal{E}$, we have that  $$\|\beta_l S-\gamma_l\|_{\infty}\le  N\beta_{\max}+\gamma_{\max}:=m_{\infty}\quad 
\forall\ t\ge 0, \ l=1,2.$$ 
It then follows from Lemma \ref{Harnck-lemma} that there is a positive constant $c_{*}=c_*(d_1,d_2,m_{\infty})$ such that 
\begin{equation}\label{DP1}
\|I_l(t)\|_{\infty}\le c_*\min_{j\in\Omega}I_{l,j}(t)\quad \forall\ t>1.
\end{equation}
Inequality \eqref{DP1} will be of significant to completing some of the proofs of our main results in the subsequent section.  The following basic result on the uniform persistence of the susceptible population holds.

 \begin{lem}\label{lem-0}
     Let $(S(t),I_1(t),I_2(t))$ be a solution of \eqref{model-eq1} with initial data in $\mathcal{E}$. Then ${\bf 0}<<_1 S(t)$ for all $t>0$ and 
     \begin{equation}\label{lem-0-eq}
         \liminf_{t\to\infty}\sum_{j\in\Omega}S_j(t)\ge \min\Big\{N, \frac{\gamma_{\min}}{\beta_{\max}}\Big\}.
     \end{equation}
 \end{lem}
 \begin{proof}
     Let $(S(t),I_1(t),I_2(t))$ be a solution  of \eqref{model-eq1} with initial data in $\mathcal{E}$. We distinguish two cases.
     
     \noindent{\bf Case 1.} $I_{l,j}(0)=0$ for every $l=1,2$ and $j\in\Omega$. Then $I_{l,j}(t)=0$ for every $t>0$, $l=1,2$ and $j\in\Omega$. This shows that $S(t)$ satisfies 
     $$
     \begin{cases}
     \frac{dS}{dt}=d_S\mathcal{L}S &  t>0,\cr 
     \sum_{j\in\Omega}S_j(0)=N.
     \end{cases}
     $$
 Hence, there is some $j_0\in\Omega$ such that $S_{j_0}(0)>0$. Hence ${\bf 0}<<_1 S_{j_0}(0)e^{t\mathcal{L}}B_{j_0}\le_1 S(t)$ for all $t>0$.

 \noindent{\bf Case 2.} $I_{l,j_0}(0)>0$ for some $l\in\{1,2\}$ and $j_0\in\{1,\cdots,k\}$. Then ${\bf 0}<<_1I_{l}(t)$ for every $t>0$. As a result, by \eqref{model-eq2},
 $$
  d_S\mathcal{L}\Big(e^{2N\beta_{\max}t}S\Big) +e^{2N\beta_{\max}t}\gamma_{l}\circ I_l\le_1 \frac{d(e^{2N\beta_{\max}t}S)}{dt} \quad \forall\ t>0.
 $$
 This implies that, since $\{e^{t\mathcal{L}}\}$ is strongly positive, 
  \begin{align*}
 {\bf 0}\ll_1 &\int_{0}^te^{\tau d_S\mathcal{L}}\Big(e^{2N\beta_{\max}\tau}\gamma_{l}\circ I_l(\tau)\Big)d\tau\cr 
 \le_1 & e^{2N\beta_{\max}t}e^{td_S\mathcal{L}}(S(0))+\int_{0}^te^{\tau d_S\mathcal{L}}\Big(e^{2N\beta_{\max}\tau}\gamma_{l}\circ I_l(\tau)\Big)d\tau\cr 
 \le_1& e^{2N\beta_{\max}t}S(t)\quad \forall\ t>0.
 \end{align*}
 Hence ${\bf 0}<<_1S(t)$ for every $t>0$. It remains to show that \eqref{lem-0-eq} holds. Observe  from \eqref{Eq1:1} that 
 \begin{align*}
 \frac{d\sum_{j\in\Omega}S_j}{dt}= &\sum_{j\in\Omega}\sum_{l=1}^2\gamma_{l,j}I_{l,j}-\sum_{j\in\Omega}S_j\sum_{l=1}^2\beta_{l,j}I_{l,j}\cr 
 \ge & \gamma_{\min}\sum_{j\in\Omega}\sum_{l=1}^2I_{l,j}-\beta_{\max}\sum_{j\in\Omega}S_j\sum_{l=1}^2I_{l,j}\cr 
 =& \gamma_{\min}(N-\sum_{j\in\Omega}S_j)-\beta_{\max}\sum_{j\in\Omega}S_j\sum_{l=1}^2I_{l,j}\cr 
 \ge &  \gamma_{\min}(N-\sum_{j\in\Omega}S_j)-\beta_{\max}\sum_{j\in\Omega}S_j(N-\sum_{j\in\Omega}S_j)\cr
 =&\gamma_{\min}\Big(\frac{\gamma_{\min}}{\beta_{\max}}-\sum_{j\in\Omega}S_j\Big)\Big(N-\sum_{j\in\Omega}S_j\Big).
 \end{align*}
 Therefore, since ${\bf 0}<<_1S(t)$ for every $t>0$, we deduce from the last inequality that \eqref{lem-0-eq}.
 \end{proof}

The following results hold.

\begin{lem}\label{lem1} Fix $l=1,2$ and suppose that $\mathcal{R}_{0,l}(N)>1$.  There is $\delta_l^*>0$  such that 
\begin{equation}\label{Z10}   
\limsup_{t\to\infty}\Big\|S(t)-\frac{N}{k}{\bf 1}\Big\|\ge \delta_l^*
\end{equation} 
for any solution $(S(t),I_1(t),I_2(t))$  of \eqref{model-eq1} with a initial data in $\mathcal{E}$ satisfying $\|I_l(0)\|_{\infty}>0$. 
\end{lem}
\begin{proof} First, note that $\lambda_*(d_l\mathcal{L}+{\rm diag}(\frac{N}{k}\beta_l-\gamma_l))>0$ since $\mathcal{R}_{0,l}(N)>1$.    We claim that \eqref{Z10} works with $\delta_l^*:=\frac{\lambda_*(d_l\mathcal{L}+{\rm diag}(\frac{N}{k}\beta_l-\gamma_l))}{4\beta_{\max}}$. Indeed, suppose by contradiction this  is false. Hence, there is some  initial data $(S(0),I_1(0),I_2(0))\in \mathcal{E}$ with $\|I_l(0)\|_{\infty}>0$ such that 
\begin{equation*}  
\Big\|S(t)-\frac{N}{k}{\bf 1}\Big\|\le2\delta_l^*\quad \forall\ t\ge 0.
\end{equation*}
This in turn together with \eqref{model-eq2} imply that the function $(\underline{I}_1(t),\underline{I}_2(t))=(e^{\tilde{\delta}_lt}I_1(t),e^{\tilde{\delta}_lt}I_2(t))$  satisfies
\begin{equation}\label{Sec2:Eq3}    
\begin{cases}     
\frac{d\underline{I}_1}{dt}\ge d_1\mathcal{L}\underline{I}_1+\Big(\gamma_1-\frac{N}{k}\beta_1\Big)\circ \underline{I}_1 & t>0,\cr     
\frac{d \underline{I}_2}{dt}\ge d_2\mathcal{L}\underline{I}_{2}+\Big(\gamma_2-\frac{N}{k}\beta_2\Big)\circ \underline{I}_2 & t>0.    
\end{cases}\end{equation}
where $\tilde{\delta}_l:=2\delta_l^*\beta_{\max}. $  Let $E_l$ be a positive eigenvector associated with $\lambda_*(d_l\mathcal{L}+{\rm diag}(\frac{N}{k}\beta_l-\gamma_l))$. Since ${\bf 0}<<_1I_{l}(1)$, then there is some $\eta_l>0$ such that $  \eta_lE_l<<_1e^{\overline{\delta}}I_{l}(1)$. It then follows from the monotonicity of the semiflow generated by the matrix semigroup $\{e^{t(d_l\mathcal{L}+{\rm diag}(\frac{N}{k}\beta_l-\gamma_l))}\}_{t\ge 0}$ and \eqref{Sec2:Eq3} that 
$$ \underline{I}_l(t)\ge\eta_l e^{t\lambda_*(d_l\mathcal{L}+{\rm diag}(\frac{N}{k}\beta_l-\gamma_l))}E_l\quad \forall\ t>0,$$
which implies 
$$ I_l(t)\ge\eta_l e^{t(\lambda_*(d_l\mathcal{L}+{\rm diag}(\frac{N}{k}\beta_l-\gamma_l))-\tilde{\delta})}E_l\quad \forall\ t>0.$$
As a result, $\sum_{j\in\Omega}I_{l,j}(t)\to \infty$. This  contradicts with \eqref{Eq1:1}, so we deduce that the desired result hold. 
\end{proof}

\begin{lem}\label{lem2} Fix $l=1,2$ and  suppose that $\mathcal{R}_{0,l}(N)>1$. Then there is $\sigma_l^*>0$ such that 
\begin{equation}\label{Z5}
    \limsup_{t\to\infty}\sum_{j\in\Omega}\sum_{p=1}^2I_{p,j}(t)\ge \sigma_l^*
\end{equation} for every solution $(S(t),I_1(t),I_2(t))$  of \eqref{model-eq1} with an data in $\mathcal{E}$ satisfying $\|I_l(0)\|_{\infty}>0$.

\end{lem}
\begin{proof} Let $(S(t),I_1(t),I_2(t))$ be a solution  of \eqref{model-eq1} with an initial data satisfying $\|I_l(0)\|_{\infty}>0$ and set
$$
\overline{I}^*:=\limsup_{t\to\infty}\sum_{j\in\Omega}\sum_{p=1}^2I_{p,j}(t).
$$
Observe that
\begin{align*}
    \frac{d\widehat{S}_j}{dt}=&d_S\sum_{i\in\Omega}(L_{j,i}S_i-L_{i,j}S_j)+\sum_{q=1}^2(\gamma_{q,j}I_{q,j}-\beta_{q,j}I_{q,j}S_j)\cr
    &-\frac{1}{k}\sum_{k\in\Omega}\left(d_S\sum_{p\in\Omega}(L_{k,p}S_p-L_{p,k}S_k)+\sum_{q=1}^2(\gamma_{q,k}I_{q,k}-\beta_{q,k}I_{q,k}S_k)\right)\cr 
    =&d_S\sum_{i\in\Omega}(L_{j,i}S_i-L_{i,j}S_j)+\sum_{q=1}^2(\gamma_{q,j}I_{q,j}-\beta_{q,j}I_{q,j}S_j)\cr
    &-\frac{1}{k}\left(d_S\sum_{k\in\Omega}\sum_{p\in\Omega}(L_{k,p}S_p-L_{p,k}S_k)+\sum_{k\in\Omega}\sum_{q=1}^2(\gamma_{q,k}I_{q,k}-\beta_{q,k}I_{q,k}S_k)\right)\cr
    =&d_S\sum_{i\in\Omega}(L_{j,i}S_i-L_{i,j}S_j)+\sum_{q=1}^2(\gamma_{q,j}I_{q,j}-\beta_{q,j}I_{q,j}S_j)-\frac{1}{k}\sum_{k\in\Omega}\sum_{q=1}^2(\gamma_{q,k}I_{q,k}-\beta_{q,k}I_{q,k}S_k)\cr
    =&d_S\sum_{i\in\Omega}(L_{j,i}S_i-L_{i,j}S_j)+\sum_{q=1}^2\Big(\widehat{\gamma_{q,j}I_{q,j}}-\widehat{\beta_{q,j}I_{q,j}S_j}\Big)\cr
    =&d_S\sum_{i\in\Omega}\Big(L_{j,i}\Big(\widehat{S}_i+\frac{1}{k}\sum_{k\in\Omega}S_k\Big)-L_{i,j}\Big(\widehat{S}_j+\frac{1}{k}\sum_{k\in\Omega}S_{k}\Big)\Big)+\sum_{q=1}^2\Big(\widehat{\gamma_{q,j}I_{q,j}}-\widehat{\beta_{q,j}I_{q,j}S_j}\Big)\cr
    =&d_S\sum_{i\in\Omega}\Big(L_{j,i}\widehat{S}_i-L_{i,j}\widehat{S}_j+(L_{j,i}-L_{i,j})\frac{1}{k}\sum_{k\in\Omega}S_k\Big)+\sum_{q=1}^2\Big(\widehat{\gamma_{q,j}I_{q,j}}-\widehat{\beta_{q,j}I_{q,j}S_j}\Big)\cr
    =&d_S\sum_{i\in\Omega}\Big(L_{j,i}\widehat{S}_i-L_{i,j}\widehat{S}_j\Big)+\sum_{q=1}^2\Big(\widehat{\gamma_{q,j}I_{q,j}}-\widehat{\beta_{q,j}I_{q,j}S_j}\Big),
\end{align*} 
since $\tilde{L}$ is symmetric. Thus,
\begin{equation}\label{Z6}
    \frac{d\widehat{S}}{dt}=d_S\mathcal{L}\widehat{S}+\sum_{q=1}^2\Big(\widehat{\gamma_q\circ I_{q}}-\widehat{\beta_q\circ{S}\circ{I}_q}\Big)\quad t>0.
\end{equation}
As a result, by the variation of constant formula,

\begin{equation}\label{Z7}
    \hat{S}(t+t')=e^{td_S\mathcal{L}}\hat{S}(t')+\int_0^te^{(t-\sigma)d_S\mathcal{L}}\sum_{q=1}^2\Big(\widehat{\gamma_q\circ I_{q}}-\widehat{\beta_q\circ{S}\circ{I}_q}\Big)(t'+\sigma)d\sigma\quad \forall\ t>0, \ t'\ge 0.
\end{equation}
Hence, by \eqref{Z3} and \eqref{Z8}, since $<\widehat{U},Q_1>=0$ for every $U\in\mathbb{R}^k$, then for every $t>0$ and $t'\ge 0$. 
\begin{align*}
    \|\widehat{S}(t+t')\|\le&  e^{td_S\tilde{d}_*}\|\widehat{S}(t')\|+\int_{0}^{t}e^{(t-\sigma)\tilde{d}_*d_S}\sum_{q=1}^2\Big\|\Big(\widehat{\gamma_q\circ I_{q}}-\widehat{\beta_q\circ{S}\circ{I}_q}\Big)(t'+\sigma)\Big\|d\sigma\cr
    \le &2  e^{t\tilde{d}_*d_S}\|{S}(t')\|+2\int_{0}^{t}e^{(t-\sigma)\tilde{d}_*d_S}\sum_{q=1}^2(\|\gamma_q\circ I_q\|+\|{\beta_q}\circ{S}\circ I_q\|)(t'+\sigma)d\sigma\cr
    \le &2 \sqrt{k}N e^{t\tilde{d}_*d_S}+2\int_{0}^{t}e^{(t-\sigma)\tilde{d}_*d_S}\sum_{q=1}^2(\|\gamma_q\|+\|{\beta_q}\|\|S\|_{\infty})\|{I}_q(t'+\sigma)\|_{\infty}d\sigma\cr
    \le & 2 \sqrt{k}N e^{t\tilde{d}_*d_S}+2\sqrt{k}(\gamma_{\max}+\beta_{\max}N)\int_{0}^{t}e^{(t-\sigma)\tilde{d}_*d_S}\sum_{q=1}^2\sum_{j\in\Omega}{I}_{q,j}(t'+\sigma)d\sigma. 
\end{align*}
Taking limsup as  $t'\to\infty$ in the last inequality, we get 
\begin{align*}
    \limsup_{\tau\to\infty}\|\widehat{S}(\tau)\|\le & 2 \sqrt{k}N e^{t\tilde{d}_*d_S}+2\overline{I}^*\sqrt{k}(\gamma_{\max}+\beta_{\max}N)\int_{0}^{t}e^{(t-\sigma)\tilde{d}_*d_S}d\sigma\cr 
    \le&  2 \sqrt{k}N e^{t\tilde{d}_*d_S}+\frac{2\overline{I}^*\sqrt{k}(\gamma_{\max}+\beta_{\max}N)}{|\tilde{d}_*|d_S}\quad \forall\ t>0,
\end{align*}
which implies that 
\begin{equation}\label{Z9}
    \limsup_{\tau\to\infty}\|\widehat{S}(\tau)\|\le\frac{2\overline{I}^*\sqrt{k}(\gamma_{\max}+\beta_{\max}N)}{|\tilde{d}_*|d_S}:=\eta_*\overline{I}^*,
\end{equation}
where $\eta_*= \frac{2\sqrt{k}(\gamma_{\max}+N\beta_{\max})}{|\tilde{d}_*|d_S} $. 
Recalling that $Q_1=\frac{N}{\sqrt{k}}{\bf 1}$, then by \eqref{Eq1:1} and \eqref{Z4},
\begin{align*}
\widehat{S}=S-\frac{\sum_{j\in\Omega}S_j}{k}{\bf 1} = S-\frac{N}{k}{\bf 1}+\frac{\sum_{j\in\Omega}\sum_{l=1}^2I_{l,j}}{k}{\bf 1}
\end{align*}
which yields 
$$
\Big\|S-\frac{N}{k}{\bf 1}\Big\|\le \|\widehat{S}\|+\frac{\sum_{i\in\Omega}\sum_{l=1}^2I_{l,i}}{k}\|\tilde{{\bf E}}\|=\|\widehat{S}\|+\frac{1}{\sqrt{k}}\sum_{i\in\Omega}\sum_{l=1}^2I_{l,i}.
 $$
 Hence, by \eqref{Z10} and \eqref{Z9}, we derive that 
 $$ 
 \delta_l^*\le(\eta_*+\frac{1}{\sqrt{k}})\overline{I}^*, 
 $$
 which implies that \eqref{Z5} holds for $\sigma_l^*=\frac{\delta_l^*}{\eta_*+\frac{1}{\sqrt{k}}}$.
\end{proof}

\section{Proofs of the Main Results}\label{Sec4}

\subsection{Proof of Theorem \ref{T1}}

We give a proof of Theorem \ref{T1}. 

\begin{proof}{\rm (i)} Let $(S(t),I_1(t),I_2(t))$ be a solution of \eqref{model-eq1} with initial data in $\mathcal{E}$. Thanks to Lemma \ref{lem-0}, we know that ${\bf 0}\ll_1S(t)$ for all $t>0$ and that 
$$ 
\liminf_{t\to\infty}\sum_{j\in\Omega}S_j(t)\ge \eta_0:=\min\left\{N,\frac{\gamma_{\min}}{\beta_{\max}}\right\}.
$$
Hence there is $t_0>0$ such that 
\begin{equation*}
   \frac{\eta_0}{2}\le \sum_{j\in\Omega}S_j(t)\quad \forall\ t\ge t_0.
\end{equation*}
Hence, for every $t\ge t_0$, there is $j_t\in\Omega$ such that 
\begin{equation*}\label{S-01}
   \frac{\eta_0}{2k}B_{j_t}\le_1 S(t).
\end{equation*}
But, 
$$
d_S\mathcal{L}S-2N\beta_{\max}S\le_1 \frac{dS}{dt}  \quad \forall\ t>0.
$$
Therefore, in view of \eqref{S-01} and the monotonicity of $\{e^{t\mathcal{L}}\}$, we have that 
$$
\frac{\eta_0}{2k}e^{-2N\beta_{\max}} e^{d_S\mathcal{L}}B_{j_t}\le_1 S(t+1) \quad \forall\ t\ge t_0.
$$
Therefore, setting ${B}_{j}^{d_S}:=e^{d_S\mathcal{L}}B_j$ for each $j=1,\cdots,k$  and ${B}^{*,d_S}:=(\min_{j\in\Omega}{B}_{j,1}^{d_S},\cdots, \min_{j\in\Omega}{B}_{j,1}^{d_S})^T$, then 
$$ 
\frac{\eta_0}{2k}{B}^{*,d_S}\le_1 S(t)\quad \forall\ t\ge t_0+1,
$$
which implies that \eqref{T1-e0} holds since $ {\bf 0}\ll_1  {B}^{*,d_S}$ by \eqref{Z8-2-2}, and ${B}^{*,d_S}$ is independent of the initial data.

\medskip

\quad {\rm (ii)} The fact that the DFE is linearly stable when $\mathcal{R}_0(N)<1$ and unstable when $\mathcal{R}_0(N)>1$ follows from standard results, (see for example \cite[Theorem 2]{DW2002}). Next, we proceed to prove assertions {\rm (ii-1)-(ii-3)}.

\noindent {\rm (ii-1)} First, suppose that $\mathcal{R}_{0,l}(N)< \frac{1}{k}$ for some $l=1,2$. 
Hence, in view of \eqref{R_0-l}, we have that 
$$
N\rho\big(\mathcal{F}_l\mathcal{V}_l^{-1}\big)< 1,
$$
which is equivalent to 
\begin{equation*}
   \lambda_l:= \lambda_*\big(d_l\mathcal{L}+{\rm diag}\big(N\beta_l-\gamma_l\big)\big)<0.
\end{equation*}
Now, let $(S(t),I_1(t),I_2(t))$ be a solution of \eqref{model-eq1} with initial data in $\mathcal{E}$. From \eqref{Eq1:1} and \eqref{model-eq2}, we have that 
\begin{equation}\label{S1}
\frac{dI_l}{dt}\le_1 d_l\mathcal{L}I_l+(N\beta_{l}-\gamma_l)\circ I_l \quad \forall\ t\ge 0.
\end{equation}
Hence, if $E_{l}^*$ denotes the positive eigenvector associated with $\lambda_l$ with $\min_{j\in\Omega}E^*_{l,j}=1$, it follows from \eqref{S1} and the comparison principle that 
$$ 
I_{l}(t)\le_1 \|I_l(0)\|_{\infty}e^{t\lambda_l}E_{l}^*\to {\bf 0} \quad \text{as} \quad t\to\infty
.$$
As result, if $\mathcal{R}_0(N)<1$, then $\sum_{l=1}^2\|I_l(t)\|\to 0$ as $t\to\infty$. Moreover, it follows from \eqref{Z9} and the fact that $\sum_{j\in\Omega}S_j=N-\sum_{j\in\Omega}\sum_{l=1}^2I_{l,j}$ that 
$$
\lim_{t\to\infty}\Big\|S(t)-\frac{N}{k}{\bf 1}\Big\|=\lim_{t\to\infty}\Big\|S(t)-\frac{N-\sum_{j\in\Omega}\sum_{l=1}^lI_{l,j}}{k}{\bf 1}\Big\|=\lim_{t\to\infty}\|\widehat{S}(t)\|=0,
$$ 
which yields the desired result.

\quad Next, suppose that $\mathcal{R}_{0,l}=\frac{1}{k}$ for some $l=1,2$. 
For every $t\ge 0$, let $c_{l}(t)$ be given by 
$$
c_l(t)=\inf\{c>0\ :\ I_{l}(t)\le_1 c E^*_l\}.
$$
Hence, since \eqref{S1} holds,  by the comparison principle, we deduce that $c_l(t)$ is nonincreasing in $t\ge 0$. Thus 
$$ 
c_l^{\infty}=\lim_{t\to\infty}c_l(t)=\inf_{t\ge 0}c_l(t)\in[0,c_l(0)].
$$
It is clear  from the definition of $c_l(t)$ that 
\begin{equation}\label{S3}
    I_l(t)\le_1 c_l(t)E^*_l \quad \forall\ t\ge 0.
\end{equation}
We claim that 
\begin{equation}\label{S2}
    c_l^{\infty}=0.
\end{equation}

Indeed, since $\min_{j\in\Omega}E^*_{l,j}=1$, it follows from the definition of $c_l(t)$ that  for each $t\ge 0$, there is $j_t\in\Omega$ such that
\begin{equation}\label{S5}
   I_{l,j_t}(t)\ge c_l(t)E^*_{l,j_t}\ge c_{l}(t) \ge c_l^{\infty},
\end{equation}
which implies that 
$$
\sum_{j\in\Omega}\sum_{p=1}^2I_{p,j}(t)\ge c_l^{\infty}\quad \forall\ t\ge 0.
$$
Therefore, by \eqref{Eq1:1},
$$
\sum_{j\in\Omega}S_j(t)=N-\sum_{j\in\Omega}\sum_{p=1}^2I_{p,j}(t)\le N-c_{l}^{\infty}\quad \forall\ t\ge 0.
$$
This in turn implies that $S(t)\le (N-c_{l}^{\infty}){\bf 1}$ for all $t\ge 0$. Therefore,
\begin{equation*}
\frac{dI_p}{dt}\le_1 d_p\mathcal{L}I_p+((N-c_l^{\infty})\beta_{p}-\gamma_l)\circ I_p\le_1 d_p\mathcal{L}I_p-c_l^{\infty}\beta_{\min}I_p +(N\beta_p-\gamma_p)\circ I_p \quad \forall\ t\ge 0, 
 p=1,2.
\end{equation*}
Therefore, 
$$
I_p(t)\le e^{-tc_l^{\infty}\beta_{\min}}\|I_p(0)\|_{\infty}E^*_{p}, \quad p=1,2, \ t\ge 0.
$$
This together with \eqref{S5} implies that 
$$
c_l^{\infty}\le e^{-tc_{l}^{\infty}\beta_{\min}}\|I_l(0)\|_{\infty}\|E^*_l\|_{\infty}\quad \forall\ t\ge 0,
$$
which implies that $c_l^{\infty}=0$. So \eqref{S2} holds.  Now, from \eqref{S2} and \eqref{S3} we have that $\|I_l(t)\|\to 0$ as $t\to\infty$. 
Therefore, if $\mathcal{R}_0(N)=\frac{1}{k}$, we can now proceed as in the previous case to establish that  $\sum_{l=1}^2I_l(t)\|\to 0$ and  $\Big\|S(t)-\frac{N}{k}{\bf 1}\Big\|\to 0$ as $t\to\infty$. This completes the proof of {\rm (ii-1)}.

\medskip

\noindent{\rm (ii-2)} Suppose that $\mathcal{R}_{0,l}(N)<1$ for some $l=1,2$.  Then $\lambda_*(d_l\mathcal{L}+{\rm diag}(\frac{N}{k}\beta_l-\gamma_l))<0$. Therefore, by the continuity of the principal eigenvalue with respect to parameters, there is $\varepsilon_l>0$ such that 
\begin{equation*}
   \lambda^*_l:= \lambda_*\Big(d_l\mathcal{L}+{\rm diag}\Big(\frac{(N+\varepsilon_l)}{k}\beta_l-\gamma_l\Big)\Big)<0.
\end{equation*}
Let $E_l^*$ denote the eigenvector of $\lambda_l^*$ with $\min_{j\in\Omega}E^*_{l,j}=1$ and set
\begin{equation}\label{d-l-expression}
    d^l:=\frac{3Nk\sqrt{k}(\gamma_{\max}+N\beta_{\max})}{|\tilde{d}_*|\varepsilon_l}.
\end{equation}
We now show that the desired result holds for any $d_S>d^l$. So fix $d_S>d^l$. Let $(S(t),I_1(t),I_2(t))$ be solution of \eqref{model-eq1} with a positive initial data in $\mathcal{E}$.  Since $\sum_{j\in\Omega}\sum_{l=1}^lI_{l,j}(t)\le N$ for every $t\ge 0$, it follows from \eqref{Z9} that 
$$
\limsup_{t\to\infty}\Big\|S(t)-\frac{\sum_{j\in\Omega}S_j(t)}{k}{\bf 1}\Big\|\le\frac{2N\sqrt{k}(\gamma_{\max}+\beta_{\max}N)}{|\tilde{d}_*|d_S}=\frac{2d^l\varepsilon_l}{3kd_S}<\frac{2}{3k}\varepsilon_l.
$$
Recalling that $\sum_{j\in\Omega}S_j=N-\sum_{j\in\Omega}\sum_{p=1}^2I_{p,j}$, then there is $t_1\gg 1$ such that 
$$
S(t)-\frac{\big(N-\sum_{j\in\Omega}\sum_{p=1}^2I_{p,j}\big)}{k}{\bf 1}\le_1 \frac{2\varepsilon_l}{3k}{\bf 1}\quad \forall\ t>t_1.
$$
This implies that 
\begin{align*}
\frac{dI_l}{dt}\le_1 & d_l\mathcal{L}I_l+\Big(\frac{\big(N+\varepsilon_l-\sum_{j\in\Omega}\sum_{p=1}^2I_{p,j}\big)}{k}\beta_l-\gamma_l\Big)\circ I_l\cr 
\le_1 &
d_l\mathcal{L}I_l+\Big(\frac{(N+\varepsilon_l)}{k}\beta_l-\gamma_l\Big)\circ I_l\quad \forall\ t>t_1.
\end{align*}
Hence, by the comparison principle,
$$
I_l(t)\le_1\|I_l(t_1)\|_{\infty}e^{(t-t_1)\lambda_l^*}E_l^*\to {\bf 0}\quad \text{as}\ t\to\infty.
$$
In particular, if $\mathcal{R}_0(N)<1$, taking $d^*=\max\{d^1,d^2\}$, where $d^l$ is given by \eqref{d-l-expression}. We have that $\sum_{l=1}^2\|I_l(t)\|\to 0$ as $t\to\infty$ for every $d_S>d^*$. Moreover, we can proceed as in the proof of {\rm (ii-1)} to prove that $(S(t),I_1(t),I_2(t))\to (\frac{N}{k}{\bf 1},{\bf 0},{\bf 0})$ as $t\to\infty$ for every $d_S>d^*$.

\medskip

\noindent {\rm (ii-3)} Suppose that $\mathcal{R}_{0,l}(N)>1$ for each $l=1,2$. 
 We proceed in four steps. 

\noindent{\bf Step 1.} In the current step, we show that there is a positive constant $m_{\rm up}>0$ such that 
\begin{equation}\label{A1-1}
\liminf_{t\to\infty}\sum_{j\in\Omega}\sum_{l=1}^2I_{l,j}(t)\ge m_{\rm up}
\end{equation}
for any solution $(S(t),I_1(t),I_2(t))$ of \eqref{model-eq1} with a positive initial in $\mathcal{E}$.  To see this, first recall that the set  $\mathcal{E}$ is compact, and invariant for the semiflow generated by classical solution of \eqref{model-eq1}. Now define the mapping $\Xi : \mathcal{E}\to [0,\infty)$ by
$$ 
\Xi(S,I_1,I_2)=\sum_{j\in\Omega}\sum_{l=1}^2I_{l,j}\quad \forall\ (S,I_1,I_2)\in\mathcal{E}.
$$
Clearly, the mapping $\Xi$ is continuous. Furthermore, $\Xi(S(t),I_1(t),I_2(t))>0$ for every $t>0$ whenever $(S(0),I_1(0),I_2(0))\in\mathcal{E}$ and satisfies $\Xi((S(0),I_1(0),I_2(0))>0$. Note from \eqref{Z5} that 
$$ 
\limsup_{t\to\infty}\Xi(S(t),I_1(t),I_2(t))\ge \min\{\sigma_1^*,\sigma_2^*\}\quad \forall\ (S(0),I_1(0),I_2(0))\in \mathcal{E}, \ \Xi(S(0),I_1(0),I_2(0))>0.
$$ 
Then  applying   persistence theory \cite[Theorem 5.2]{ST2011}, we deduce that  \eqref{A1-1} holds.

\noindent{\bf Step 2.} In the current step, we show that there is a positive $m_{\rm low}>0$ constant such that 
\begin{equation}
\liminf_{t\to\infty}\min_{j\in\Omega}\sum_{l=1}^2I_{l,j}(t)\ge m_{\rm low}
\end{equation}
for any solution $(S(t),I_1(t),I_2(t))$ of \eqref{model-eq1} with a positive initial. To this end, let $(S(t),I_1(t),I_2(t))$ be a solution of \eqref{model-eq1} with a positive initial data. Note that
\begin{equation*}
       d_l\mathcal{L}I_l-\gamma_{\max}I_l\le_1  \frac{dI_l(t)}{dt}\quad t>0,\ l=1,2,
\end{equation*}
which,  thanks to the positivity of $\{e^{t\mathcal{L}}\}_{t\ge 0}$,  implies that 
\begin{equation*}
      e^{-\gamma_{\max}t}e^{d_lt\mathcal{L}}I_l(t')\le_1  I_l(t+t') \quad t>0, \ t'>0,\ l=1,2.
\end{equation*}
Thus, setting $\tilde{I}(t)=\sum_{l=1}^2I_l(t)$, we obtain
\begin{equation}\label{A1-2}
e^{-\gamma_{\max}t}e^{\min\{d_1,d_2\}t}\tilde{I}(t')\le_1 \tilde{I}(t+t')\quad\forall\  t>0,\ t'>0.
\end{equation}
Next, thanks to \eqref{A1-1}, there is $t_0'>0$ such that 
\begin{equation}\label{DP2}
\sum_{j\in\Omega}\tilde{I}_j(t')\ge \frac{m_{\rm up}}{2}\quad \forall\ t'\ge t_0'.
\end{equation}
By  \eqref{DP2} and  \eqref{DP1}, 
$$
\frac{m_{\rm up}}{2}\le \sum_{j\in\Omega}\sum_{l=1}^2I_{l,j}(t)\le k\sum_{l=1}^2\|I_l(t)\|_{\infty}\le kc_*\sum_{l=1}^2\min_{j\in\Omega}I_{l,j}(t)\quad \forall\ t>t_0'.
$$
As a result, 
$$
\liminf_{t\to\infty}\sum_{l=1}^2\min_{j\in\Omega}I_{l,j}(t)\ge \frac{m_{\rm up}}{2kc_*}>0.
$$
Therefore, \eqref{Z2} holds for $m_{\rm low}:=\frac{m_{\rm up}}{2kc_*}$.

\noindent {\bf Step 3.} Here, we show that 
\begin{equation}\label{A1-5}
    \limsup_{t\to\infty}S_j(t)\le s_{\max}
\end{equation}
for any solution $(S(t),I_1(t),I_2(t))$ of \eqref{model-eq1} with a positive initial data. Let $(S(t),I_1(t),I_2(t))$ be a solution of \eqref{model-eq1} with a positive initial data. Thanks to \eqref{T1-e1}, without loss of generality, we may suppose that 
\begin{equation}\label{A1-6}
    \min_{j\in\Omega}\sum_{l=1}^2I_{l,j}(t)\ge \frac{m_*}{2}\quad \forall\ t\ge 0.
\end{equation}
Observe from \eqref{model-eq2} that $S(t)$ satisfies
\begin{equation*}
    \frac{dS(t)}{dt}\le_1 d_S\mathcal{L}S+\sum_{l=1}^2(S_{\max}-S)\circ\beta_l\circ I_l\quad t>0,
\end{equation*}
where $S_{\max}=(s_{\max},\cdots,s_{\max})^T$. Next set
$$
\overline{S}(t)=S_{\max}+e^{-\frac{m_*\beta_{\min}}{2}t}S(0)\quad \forall\ t\ge 0.
$$
 Noting from  \eqref{A1-6} that
 $$
 {\bf 0}\le_1\Big(\beta_{\min}\min_{j\in\Omega}\sum_{l=1}^2I_{l,j}-\frac{m_*\beta_{\min}}{2}\Big){\bf 1}\le \sum_{l=1}^2\beta_{l}\circ I_l-\frac{m_*\beta_{\min}}{2}{\bf 1}, \quad\ t>0,
 $$
 then
\begin{align*}
    {\bf 0}\le_1 & e^{-\frac{m_*\beta_{\min}}{2}t}\Big(\sum_{l=1}^2\beta_l\circ I_l-\frac{m_*\beta_{\min}}{2}{\bf 1}\Big)\circ S(0)\cr 
    =& \frac{d\overline{S}(t)}{dt}-d_S\mathcal{L}\overline{S}(t)-\sum_{l=1}^2(S_{\max}-\overline{S}(t))\circ\beta_l\circ I_l\quad \forall\ t>0.
\end{align*}
It then follows from the positivity of $\{e^{t\mathcal{L}}\}_{t\ge 1}$ and the fact that $\overline{S}(0)>S(0)$ that 
$$
S(t)\le_1 \overline{S}(t)=S_{\max}+e^{-\frac{m_*\beta_{\min}}{2}t}S(0)\quad t>0,
$$
from which we deduce that $\limsup_{t\to\infty}\max_{j\in\Omega}S(t)\le \max_{j\in\Omega}S_{\max}=s_{\max}$.

\noindent{\bf Step 4.} Finally, we show that \begin{equation}
    \liminf_{t\to\infty}S_j(t)\ge s_{\min}
\end{equation}
for any solution $(S(t),I_1(t),I_2(t))$ of \eqref{model-eq1} with a positive initial data. We proceed as in step 3. This time we note that \begin{equation*}
   d_S\mathcal{L}S+\sum_{l=1}^2(S_{\min}-S)\circ\beta_l\circ I_l\le_1 \frac{dS(t)}{dt} \quad t>0,
\end{equation*}
where $S_{\min}=(s_{\min},\cdots,s_{\min})^T$, and that the function 
$$ 
\underline{S}(t)=S_{\min}-M_0e^{-N\beta_{\max}t}\quad t\ge 0,
$$
where $M_0\in \mathbb{R}_+^k$ is chosen such that $S(0)>S_{\min}-M_0$, 
satisfies $\underline{S}(0)\le S(0)$ and 
\begin{align*}
    \frac{d\underline{S}(t)}{dt}-d_S\mathcal{L}\underline{S}(t)-\sum_{l=1}^2(S_{\min}-\underline{S}(t))\circ\beta_l\circ I_l=& e^{-\frac{\beta_{\min}m_*}{2}t}\Big(\frac{\beta_{\min}m_*}{2}M_0-\sum_{l=1}^2M_0\circ\beta_l\circ I_l\Big)\cr 
    \le_1 & e^{-\frac{\beta_{\min}m_*}{2}t}(\frac{\beta_{\min}m_*}{2}-\frac{\beta_{\min}m_*}{2})M_0\cr 
    = & {\bf 0} \quad \forall\ t>0.
\end{align*}
Thus, we deduce that $S(t)\ge \underline{S}(t)$ for all $t\ge 0$, which implies that $\liminf_{t\to\infty}\min_{j\in\Omega}S(t)\ge \min_{j\in\Omega}S_{\min}=s_{\min}$. This completes the proof of {\rm (iii)}.
\end{proof}

Thanks to the proof of Theorem \ref{T1}-{\rm (iii)}, we can state the following result on the persistence of the disease for a single strain model.

\begin{prop}\label{prop1} Suppose that $\mathcal{R}_{0,1}(N)>1$. Then  there is a positive number $m_{1,*}>0$ such that for every solution $(S(t),I_{1}(t),I_2(t))$ with a  initial data in $\mathcal{E}$ satisfying $\|I_1(0)\|>0$ and $\|I_2(0)\|_{\infty}=0$, 
\begin{equation}\label{P1-e1}
    \liminf_{t\to \infty}\min_{j\in\Omega}I_{1,j}(t)\ge m_{1,*}.
\end{equation}
Furthermore,
\begin{equation}\label{P1-e2}
    \limsup_{t\to\infty}\max_{j\in\Omega}S_j(t)\le r_{1,\max} \quad \text{and}\quad \liminf_{t\to\infty}\min_{j\in\Omega}S_j(t)\ge r_{1,\min}.
\end{equation}
    
\end{prop}

\subsection{Proofs of Theorems \ref{T2} and \ref{T3}}

\begin{lem}\label{lemA-1} Suppose that $\mathfrak{R}_{1,j}=\mathfrak{R}_{1,i}$ for each $i,j\in\Omega$, $\min_{l=1,2}\mathcal{R}_{0,l}(N)>1$, and $\mathfrak{R}_{1,\min}>\mathfrak{R}_{2,\max}$. Then, there exists $\eta_1>0$ such that 
\begin{equation}\label{lemA-eq}
    \liminf_{t\to\infty}\min_{j\in\Omega}I_{1,j}\ge \eta_1
\end{equation}
    for every solution $(S(t),I_1(t),I_2(t))$ of \eqref{model-eq1} with  initial data in $\mathcal{E}$ satisfying $\|I_1(0)\|_{\infty}>0$.
\end{lem}
\begin{proof} We proceed in two steps. First, note that $\mathcal{R}_{0}(N)=\mathcal{R}_{0,1}(N)>1$. Since $r_1:=r_{1,j}$ is constant in $j\in\Omega$, and $\mathcal{R}_{0,1}>\|\mathfrak{R}_2\|_{\infty}$, then $r_1<\min_{j\in\Omega}r_{2,j}$. So, we can choose $0<\lambda_0<1$ such that  $r_1<\lambda_0\min_{j\in\Omega}r_{2,j}$. Let $\tilde{m}_*=\min\{m_*,m_{1,*}\}$ where $m_*$ and $m_{1,*}$ are given by \eqref{T1-e1} and \ref{P1-e1}, respectively. Set
\begin{equation*}
    \eta_0:=\frac{\Big(\frac{\beta_{\min}}{\lambda_0\beta_{\max}}\Big)\frac{(1-\lambda_0)\tilde{m}_*}{2}}{1+\Big(\frac{\beta_{\min}}{\lambda_0\beta_{\max}}\Big)(1-\lambda_0)}\quad \text{and}\quad m_0=\frac{\tilde{m}_*}{2}-\eta_0.
\end{equation*}
Then
\begin{equation}\label{B1-0}
    m_0>0 \quad \text{and}\quad \beta_{\min}m_0(1-\lambda_0)=\lambda_0\beta_{\max}\eta_0.
\end{equation}
{\bf Step 1.} We proceed by contradiction to establish  that 
\begin{equation}\label{B1-1}
    \limsup_{t\to\infty}\sum_{j\in\Omega}I_{1,j}(t)\ge \eta_0
\end{equation}
for every solution $(S(t),I_1(t),I_2(t))$ of \eqref{model-eq1} with a positive initial data satisfying $I_1\ne 0$. So, suppose that there is a solution $(S(t),I_1(t),I_2(t))$ of \eqref{model-eq1} with a positive initial data satisfying $I_1(0)>(0,\cdots,0)^T$ and
\begin{equation}\label{B1-2}
    \limsup_{t\to\infty}\sum_{j\in\Omega}I_{1,j}(t)<\eta_0.
\end{equation}
Hence, after translation in time, we may suppose that 
\begin{equation}\label{B1-3}
    \sum_{j\in\Omega}I_{1,j}(t)\le \eta_0\quad \forall\ t\ge 0.
\end{equation}
Observe that $S(t)$ satisfies 
\begin{align}\label{B1-4}
  d_S\mathcal{L}S -\beta_{\max}\eta_0 S+(r_2- S)\circ\beta_2\circ I_2\le_1 d_S\mathcal{L}S +(\gamma_1-\beta_1\circ S)\circ I_1+(\gamma_2-\beta_2\circ S)\circ I_2 
   =  \frac{dS(t)}{dt},
\end{align}
where $r_{2,j}=\frac{\gamma_{2,j}}{\beta_{2,j}}$ for each $j\in\Omega$. Now, thanks to \eqref{T1-e1}, \eqref{P1-e1} and \eqref{B1-3}, we have that $\|I_{2}(0)\|_{\infty}>0$ and there is $t_0>0$ such that 
$$
I_{2,j}(t)\ge m_0\quad \forall\ t\ge t_0.
$$
Next, define 
$$ 
\underline{S}(t)=\lambda_0 R_{2,\min}-e^{-m_0\beta_{\min}t}M_0\quad \forall\ t\ge  t_0,
$$
where ${\bf 0}<<_1M_0\in [\mathbb{R}_+]^k$ is fixed and satisfies $\lambda_0R_{2,\min}-e^{-m_0\beta_{\min}t_0}M_2<S(t_0)$.
 Thanks to \eqref{B1-0}, we have 
\begin{align*}
   & \frac{d\underline{S}}{dt}-d_S\mathcal{L}\underline{S}+\beta_{\max}\eta_0\underline{S}-(r_2-\underline{S})\circ\beta_2\circ I_2\cr 
   =& m_0\beta_{\min}M_0e^{-m_0\beta_{\min}t}+\beta_{\max}\eta_0(\lambda_0 R_{2,\min}-M_0e^{-m_0\beta_{\min}t})-(r_{2}-\lambda_0 R_{2,\min}+M_0e^{-m_0\beta_{\min}t})\circ\beta_2\circ I_2\cr 
   =&e^{-m_0\beta_{\min}t}\Big((m_0\beta_{\min}-\beta_{\max}\eta_0)M_0-M_0\circ\beta_2\circ I_2\Big)+\lambda_0\beta_{\max}\eta_0R_{2,\min}-(r_2-\lambda_0 R_{2,\min})\circ\beta_2\circ I_2\cr
   \le & e^{-m_0\beta_{\min}t}\Big((m_0\beta_{\min}-\beta_{\max}\eta_0)M_0-\beta_{\min}m_0M_0\Big)+\beta_{\max}\eta_0\lambda_0R_{2,\min}-\beta_{\min}m_0(r_2-\lambda_0 R_{2,\min})\cr
   \le & e^{-m_0\beta_{\min}t}\Big((m_0\beta_{\min}-\beta_{\max}\eta_0)M_0-\beta_{\min}m_0M_0\Big)+\beta_{\max}\eta_0\lambda_0R_{2,\min}-\beta_{\min}m_0(R_{2,\min}-\lambda_0 R_{2,\min})\cr
   =& e^{-m_0\beta_{\min}t}\Big((m_0\beta_{\min}-\beta_{\max}\eta_0)M_0-\beta_{\min}m_0M_0\Big)+(\beta_{\max}\eta_0\lambda_0-\beta_{\min}m_0(1-\lambda_0))R_{2,\min}\cr
   =& -e^{-m_0\beta_{\min}t}\beta_{\max}\eta_0M_0<_1{\bf 0}\qquad \forall\ t>t_0.
\end{align*}
Hence, in view of \eqref{B1-4}, $S(t)\ge \underline{S}(t)$ for all $t\ge t_0$. Thus, $I_1(t)$ satisfies
\begin{equation*}
    d_1\mathcal{L}I_1(t)+(\lambda_0R_{2,\min}-R_{1,\min}-e^{-m_0\beta_{\min}t}M_0)\circ\beta_1\circ I_{1}\le_1 \frac{dI_{1}(t)}{dt} \quad \forall\ t\ge t_0.
\end{equation*}
As a result, since $\lambda_0R_{2,\min}>R_{1,\min}$, $e^{-m_0\beta_{\min}t}\to 0$ as $t\to\infty$, and $I_{1}(t_0)>>0$, then $\|I_1(t)\|_{\infty}\to \infty$ as $t\to\infty$. This clearly, contradicts the fact that $\|I_1(t)\|_{\infty}\le N$ for all $t\ge 0$. Therefore, \eqref{B1-1} must hold.
 
\medskip

\noindent{\bf Step 2.} We complete the proof of \eqref{lemA-eq}. Indeed, first recall that the set  $\mathcal{E}$ is compact, and invariant for the semiflow generated by classical solution of \eqref{model-eq1}. Now define the mapping $\Xi : \mathcal{E}\to [0,\infty)$ by
$$ 
\Xi(S,I_1,I_2)=\sum_{j\in\Omega}I_{1,j}\quad \forall\ (S,I_1,I_2)\in\mathcal{E}.
$$
Clearly, the mapping $\Xi$ is continuous. Furthermore, $\Xi(S(t),I_1(t),I_2(t))>0$ for every $t>0$ whenever $(S(0),I_1(0),I_2(0))\in\mathcal{E}$ and satisfies $\Xi((S(0),I_1(0),I_2(0))>0$. Then  applying   persistence theory \cite[Theorem 5.2]{ST2011}, it follows from \eqref{B1-1} that there is some $\eta_{0}^*>0$ such that 
\begin{equation}\label{B1-5}
    \liminf_{t\to\infty}\sum_{j\in\Omega}I_{1,j}(t)\ge \eta_0^*
\end{equation}
for every solution $(S(t),I_1(t),I_2(t))$ of \eqref{model-eq1} with a positive initial data satisfying $I_1\ne 0$. Finally, we can employ \eqref{DP1} to conclude that \eqref{lemA-eq} follows from \eqref{B1-5}.
\end{proof}

\begin{lem}\label{lemA-2} Suppose that $\min_{l=1,2}\mathcal{R}_{0,l}(N)>1$, $\mathfrak{R}_{1,j}=\mathfrak{R}_{1,i}$ for each $i,j\in\Omega$ and $\mathcal{R}_{0,1}(N)<\mathfrak{R}_{2,\min}(N)$. Then, there exists $\eta_1>0$ such that 
\begin{equation}\label{lemA-2-eq}
    \liminf_{t\to\infty}\min_{j\in\Omega}I_{2,j}\ge \eta_1
\end{equation}
    for every solution $(S(t),I_1(t),I_2(t))$ of \eqref{model-eq1} with  initial data in $\mathcal{E}$ satisfying $\|I_2(0)\|_{\infty}>0$.
\end{lem}
\begin{proof}The proof follows a proper modification of that of Lemma \ref{A1-2}, hence it is omitted.
\end{proof}

\begin{lem}\label{lemA-3} Suppose that $\mathcal{R}_{0,1}(N)>1$ and $\mathfrak{R}_{1,j}$ is constant in $j\in\Omega$. Then the strain-1 EE is linearly stable if $\mathcal{R}_{0,1}(N)>\mathcal{R}_{0,2}(N)$ and unstable if $\mathcal{R}_{0,1}(N)<\mathcal{R}_{0,2}(N)$.
    
\end{lem}
\begin{proof} Set $r_{1,\min}=\min_{j\in\Omega}r_{1,j}$. Since $\mathfrak{R}_{1,j}$ is constant in $j\in\Omega$, then $r_{1}=r_{1,\min}{\bf 1}$ and ${\bf E}_1^*:=(r_{1},(\frac{N}{k}-r_{1,\min}){\bf 1}, {\bf 0})$ is the strain-1 EE solution. Linearizing \eqref{model-eq1} at this EE, we obtain the eigenvalue problem

\begin{equation}\label{U1}
    \begin{cases}
        \lambda W=d_S\mathcal{L}W-(\frac{N}{k}-r_{1,\min})\beta_{1}\circ W + (\gamma_2-r_{1,\min}\beta_2)\tilde{W}\cr 
        \lambda \hat{W}= d_1\mathcal{L}\hat{W}+(\frac{N}{k}-r_{1,\min})\beta_1\circ W \cr 
        \lambda \tilde{W}=d_2\mathcal{L}\tilde{W}+(r_{1,\min}\beta_2-\gamma_2)\circ \tilde{W}\cr 
        0=\sum_{j\in\Omega}(W_j+\hat{W}_j+\tilde{W}_j)
    \end{cases}
\end{equation}
in $[\mathbb{R}^k]^3$.
Note that the third equation of \eqref{U1} decouples from the first two equations. Let $(\lambda,(W,\hat{W},\tilde{W}))$ be an eigenpair of \eqref{U1}. If $\tilde{W}\ne {\bf 0}$, then $\lambda$ is an eigenvalue of the symmetric and irreducible matrix $d_2\mathcal{L}+{\rm diag}(r_{1,\min}\beta_2-\gamma_2)$, hence it is a real number. If $\tilde{W}={\bf 0}$, then we have two cases. First, if $W\ne {\bf 0}$, then $\lambda$ is an eigenvalue of the symmetric and irreducible matrix $d_S\mathcal{L}-{\rm diag}((\frac{N}{k}-r_{1,\min})\beta_1)$, hence it is a real number. Next, if $W={\bf 0}$, then $\hat{W}\ne {\bf 0}$ and $\lambda$ is an eigenvalue of the symmetric and irreducible matrix $d_1\tilde{L}$, hence it is a real number. This shows that any eigenvalue of \eqref{U1} is a real number. Now, we proceed in two cases to complete the proof of the lemma.

{\bf Case 1.} Here, suppose that $\mathcal{R}_{0,1}(N)>\mathcal{R}_{0,2}(N)$ and show that ${\bf E}_1^*$ is linearly stable. Since $\mathcal{R}_{0,2}(N)=\frac{N}{k}\rho(\mathcal{F}_2\mathcal{V}_2^{-1})$ (see \eqref{R_0-l}), then $\mathcal{R}_{0,1}(N)>\frac{N}{k}\rho(\mathcal{F}_2\mathcal{V}_2^{-1})$, which is equivalent to  $
\frac{N}{k\mathcal{R}_{0,1}(N)}\rho(\mathcal{F}_2\mathcal{V}_2^{-1})<1.
$  This in turn implies that $\lambda_*(\frac{N}{k\mathcal{R}_{0,1}(N)}\mathcal{F}_2-\mathcal{V}_2)<0.$ Observing that $\mathcal{R}_{0,1}(N)=\frac{N}{kr_{1,\min}}$ (since $\mathfrak{R}_{1,j}$ is constant in $j\in\Omega$), then 
$$ 
\frac{N}{k\mathcal{R}_{0,1}(N)}\mathcal{F}_2-\mathcal{V}_2=d_2\mathcal{L}+{\rm diag}(r_{1,\min}\beta_2-\gamma_2).
$$ 
Therefore 
\begin{equation}\label{U2}
    \lambda_*\big(d_2\mathcal{L}+{\rm diag}(r_{1,\min}\beta_2-\gamma_2)\big)=\lambda_*(\frac{N}{k\mathcal{R}_{0,1}(N)}\mathcal{F}_2-\mathcal{V}_2)<0.
\end{equation}
Now, let $(\lambda, (W,\hat{W},\tilde{W}))$ be an eigenpair of \eqref{U1}. If $\tilde{W}\ne \{\bf 0\}$, then, by \eqref{U2},  $\lambda\le \lambda_*(d_2\mathcal{L}+{\rm diag}(r_{1,\min}\beta_2-\gamma_2))<0$. If $W={\bf 0}$ and $\tilde{W}=0$, then either $W\ne {\bf 0}$ or $W={\bf 0}$. If $W\ne {\bf 0}$, then 
$$
\lambda\le \lambda_*(d_S\mathcal{L}-r_{1,\min}{\rm diag}(\mathcal{R}_{0,1}(N)-1)\beta_1)\le r_{1,\min}(\mathcal{R}_{0,1}(N)-1)\beta_{\min}<0.
$$
Finally, if $W=\tilde{W}={\bf 0}$, then $\hat{W}\ne {\bf 0}$ and satisfies 
$$
\lambda\hat{W}=d_1\mathcal{L}\hat{W}\quad \text{and}\quad \sum_{j\in\Omega}\hat{W}_j=0.
$$
Then $\hat{W}$ is not strictly positive and $\lambda$ is an eigenvalue of $d_1\mathcal{L}$. It then follows from the Perron-Frobenius Theorem that $\lambda<\lambda_*(d_1\mathcal{L})=0$. In view of the above, we see that any eigenvalue of \eqref{U1} is negative, which implies that ${\bf E}_1^*$ is linearly stable.

{\bf Case 2.} Here, suppose that $\mathcal{R}_{0,1}(N)<\mathcal{R}_{0,2}(N)$ and show that ${\bf E}_1^*$ is unstable. By the similar arguments leading to \eqref{U2}, we have that $\lambda:=\lambda_*(d_2\mathcal{L}+{\rm diag}(r_{1,\min}\beta_2-\gamma_2))>0$. Let $\tilde{W}$ be a positive eigenvector associated to  $\lambda$. Since $\lambda>0=\lambda_*(d_S\mathcal{L})$, then $\lambda$ is in the resolvent set of $d_S\mathcal{L}$. Thus, there is a unique $W\in\mathbb{R}^k$ solving the first equation of \eqref{U1}. Similarly, since $\lambda>0=\lambda_*(d_1\mathcal{L})$, there is a unique $\hat{W}$ solving the second equation of \eqref{U1}. Therefore, $(W,\hat{W},\tilde{W})$ solves the first three equations of \eqref{U1}. Moreover, since $\mathcal{L}$ is symmetric, we get
$$
\lambda\sum_{j\in\Omega}(W_j+\hat{W}_J+\tilde{W}_J)=\sum_{i,j}L_{i,j}(d_S(W_{i}-W_{j})+d_1(\hat{W}_i-\hat{W}_j)+d_2(\tilde{W}_i-\tilde{W}_j))=0,
$$
which yields that $\sum_{j\in\Omega}(W_j+\hat{W}_J+\tilde{W}_J)=0$ since $\lambda>0$. Therefore, $(\lambda, (W,\hat{W},\tilde{W}))$ solves \eqref{U2} with ${\bf 0}\ll_1 \tilde{W}$. This shows that ${\bf E}_1^*$ is unstable, since $\lambda>0$.
\end{proof}
\begin{proof}[Proof of Theorem \ref{T2}] Suppose that $\min_{l=1,2}\mathcal{R}_{0,1}(N)>1$ and  $\mathfrak{R}_{1,j}$ is constant in $j\in\Omega$. 

\noindent{\rm (i)} If $\mathcal{R}_{0,1}(N)>\mathcal{R}_{0,2}(N)$, it follows from Lemma \ref{lemA-3} that the strain-1 EE solution, ${\bf E}^*_1=(r_{1},(\frac{N}{k}-r_{1,\min}){\bf 1}, {\bf 0})$, is is linearly stable. Next, suppose that $\mathfrak{R}_{2,\max}<\mathfrak{R}_{1,\min}$.  Let $(S(t),I_1(t),I_2(t))$ be solution of \eqref{model-eq1} with a positive initial data in $\mathcal{E}$. By Lemma \ref{lemA-1}, since $\|I_{1}(0)\|_{\infty}>0$, then there is $t_1>0$ such that 
\begin{equation}\label{U1-1}
    \min_{j\in\Omega}I_{1,j}(t)\ge \frac{\eta_1}{2}\quad \forall\ t>t_1,
\end{equation}
where $\eta_1$ is the positive number given in \eqref{lemA-eq}. Next, we claim that 
\begin{equation}\label{U1-2}
    \lim_{t\to\infty}\|S(t)-r_{1,\min}{\bf 1}\|_{\infty}=0.
\end{equation}
Suppose by contradiction that \eqref{U1-2} is false. Then there is a sequence of positive numbers $\{t_{n}\}_{n\ge 1}$ converging to infinity such that 
\begin{equation}\label{U1-1-1}
    \inf_{n\ge1 }\|S(t_n)-r_{1,\min}{\bf 1}\|_{\infty}>0.
\end{equation}
Since $\|S(t)\|_{\infty}+\|I_1(t)\|_{\infty}+\|I_2(t)\|_{\infty}\le N$ for all $t\ge 0$, by the Arzela-Ascoli theorem, there is a subsequence $\{t_{n_1}\}_{n}$ of $\{t_n\}_{n\ge 1}$ and $(S^{\infty},I_1^{\infty},I_2^{\infty})\in C^1(\mathcal{R}:[\mathbb{R}_+^k]^3)$ such that $(S(t+t_{n_1}),I_1(t+t_{n_1}),I_2(t+t_{n_1}))\to (S^{\infty}(t),I_1^{\infty}(t),I_2^{\infty}(t))$ as $n\to\infty$, for $t$ on compact subsets of $\mathbb{R}$.  
Furthermore, $(S^{\infty},I^{\infty})$ satisfies 
\begin{equation}\label{U1-3}
\begin{cases}
\frac{dS^{\infty}}{dt}=d_S\mathcal{L}S^{\infty}+(r_{1,\min}{\bf 1}-S^{\infty})\circ\beta_1\circ I_1^{\infty}+(r_{2}-S^{\infty})\circ \beta_2\circ I_2^{\infty} & t\in\mathbb{R},\cr 
\frac{dI_1^{\infty}}{dt}=d_1\mathcal{L}I_1^{\infty} +(S^{\infty}-r_{1,\min}{\bf 1})\circ\beta_1\circ I_1^{\infty} & t\in\mathbb{R},\cr
\frac{dI_2^{\infty}}{dt}=d_2\mathcal{L}I_2^{\infty} +(S^{\infty}-r_{2})\circ\beta_2\circ I_2^{\infty} & t\in\mathbb{R},\cr 
N=\sum_{j\in\Omega}(S_j^{\infty}+I_{1,j}^{\infty}+I_{2,j}^{\infty}) & t\in\mathbb{R}.

\end{cases}
\end{equation}
Moreover, thanks to \eqref{T1-e2} and \eqref{U1-1},
\begin{equation}\label{U1-4}
    r_{1,\min}{\bf 1}\le_1 S^{\infty}(t)\le_1N{\bf 1},\quad {\bf 0}\le_1 I_2^{\infty}(t)\le_1 N{\bf 1}, \quad  \text{and}\quad \frac{\eta_1}{2}{\bf 1}\le_1 I_1^{\infty}(t)\le_1 N{\bf 1}\quad \forall\ t\in\mathbb{R}.
\end{equation}
Now, we claim that 
\begin{equation}\label{U1-5}
    I_2^{\infty}(t)={\bf 0}, \quad \forall\ t\in\mathbb{R}.
\end{equation}
To see this, observe from \eqref{U1-3} and the fact that $\mathcal{L}$ is symmetric that 
\begin{align*}
\frac{d\sum_{j\in\Omega}I_{2,j}^{\infty}}{dt}=&\sum_{j\in\Omega}(S_j^{\infty}-r_{1,\min})\beta_{2,j}I_{2,j}^{\infty}+\sum_{j\in\Omega}(r_{1,\min}-r_{2,j})\beta_{2,j}I_{2,j}^{\infty}\cr 
\le & \|I_2(t)\|_{\infty}\beta_{\max}\sum_{j\in\Omega}(S^{\infty}_j-r_{1,\min})+(r_{1,\min}-r_{2,\min})\beta_{\min}\sum_{j\in\Omega}I_{2,j}^{\infty},\cr 
\le & \Big(\beta_{\max}\sum_{j\in\Omega}(S^{\infty}_j-r_{1,\min})+(r_{1,\min}-r_{2,\min})\beta_{\min}\Big)\sum_{j\in\Omega}I_{2,j}^{\infty}\quad \forall\ t\in\mathbb
{R},
\end{align*}
since $r_{1,\min}<r_{2,\min}$ and $S_{\min}^{\infty}\ge r_{1,\min}$ (see \eqref{U1-4}). Therefore,
\begin{align}\label{U1-6}
    \sum_{j\in\Omega}I_{2,j}^{\infty}(t)\le&  e^{\beta_{\max}\int_{\tau}^{t}(\sum_{j\in\Omega}S^{\infty}(\sigma)-r_{1,\min})d\sigma+(t-\tau)(r_{1,\min}-r_{2,\min}))}\sum_{j\in\Omega}I_{2,j}^{\infty}(\tau)\cr 
    \le &N e^{\beta_{\max}\int_{\tau}^{t}(\sum_{j\in\Omega}S^{\infty}(\sigma)-r_{1,\min})d\sigma+(t-\tau)(r_{1,\min}-r_{2,\min}))}\quad \forall\ t>\tau.
\end{align}
On the other, since 
$$
\frac{d\sum_{j\in\Omega}I_{1,j}^{\infty}}{dt}=\sum_{j\in\Omega}(S^{\infty}_j-r_{1,\min})\beta_{1,j}I_{1,j}^{\infty},
$$
where we have used the fact that $\mathcal{L}$ is symmetric, then 
\begin{align}\label{U1-7}
N\ge \sum_{j\in\Omega}I_{1,j}^{\infty}(t+\tau)=&\sum_{j\in\Omega}I_{1,j}^{\infty}(t+\tau)+\int_{\tau}^{t}\sum_{j\in\Omega}(S_{j}^{\infty}(s)-r_{1,\min})\beta_{1,j}I_{1,j}^{\infty}(s)ds\cr 
\ge & \sum_{j\in\Omega}I_{1,j}^{\infty}(t+\tau)+\frac{\eta_1\beta_{\min}}{2}\int_{\tau}^{t}\sum_{j\in\Omega}(S_{j}^{\infty}(s)-r_{1,\min})ds \quad \forall\ t>\tau.
\end{align}
From \eqref{U1-6} and \eqref{U1-7}, we obtain that 
$$ 
0\le \|I_2^{\infty}(t)\|_{\infty}\le \sum_{j\in\Omega}I_{2,j}^{\infty}(t)\le Ne^{\frac{2N\beta_{\max}}{\eta_1\beta_{\min}}}e^{-(r_{2,\min}-r_{1,\min})(t-\tau)}\quad t>\tau.
$$
Letting $\tau\to-\infty$ in this inequality yields \eqref{U1-5}. Therefore, \eqref{U1-3} becomes 
\begin{equation}\label{U1-8}
\begin{cases}
\frac{dS^{\infty}}{dt}=d_S\mathcal{L}S^{\infty}+(r_{1,\min}{\bf 1}-S^{\infty})\circ\beta_1\circ I_1^{\infty} & t\in\mathbb{R},\cr 
\frac{dI_1^{\infty}}{dt}=d_1\mathcal{L}I_1^{\infty} +(S^{\infty}-r_{1,\min}{\bf 1})\circ\beta_1\circ I_1^{\infty} & t\in\mathbb{R},\cr 
N=\sum_{j\in\Omega}(S_j^{\infty}+I_{1,j}^{\infty}) & t\in\mathbb{R}.

\end{cases}
\end{equation}
Setting $\tilde{S}^{\infty}=S^{\infty}-r_{1,\min}{\bf 1}$, then by $\eqref{U1-4}$, ${\bf 0}\le \tilde{S}^{\infty}(t)$, for every $t\in\mathbb{R}$. Moreover, by \eqref{U1-8} and the fact that $\frac{\eta_1}{2}{\bf 1}\le I_1^{\infty} $, $\tilde{S}^{\infty}$ satisfies
$$
\frac{d\tilde{S}^{\infty}}{dt}=d_S\mathcal{L}\tilde{S}^{\infty} -\frac{\eta_1}{2}\beta_{\min}\tilde{S}^{\infty}\quad t\in\mathbb{R},
$$
which implies 
$$
{\bf 0}\le_1 \tilde{S}^{\infty}(t)\le_1 e^{-\frac{\eta_1}{2}\beta_{\min}(t-\tau)}\tilde{S}^{\infty}(\tau)\le_1 Ne^{-\frac{\eta_1}{2}\beta_{\min}(t-\tau)}{\bf 1}\to {\bf 0} \ \text{as}\ \tau\to-\infty,
$$
for any $t\in\mathbb{R}$. Hence $\tilde{S}^{\infty}={\bf 0}$ for every $t\in\mathbb{R}$, that is $S^{\infty}(t)=r_{1,\min}{\bf 1}$ for every $t\in\mathbb{R}$. As, a result, $S(t_{n_1})\to r_{1,\min}{\bf 1}$ as $n\to\infty$, which contradicts with \eqref{U1-1-1}. Therefore, \eqref{U1-2} holds.  Therefore, for $0<\varepsilon<\frac{1}{2}(r_{2,\min}-r_{1,\min})$, there is $t_{\varepsilon}>0$ such that 
$ S(t)\le_1 (r_{1,\min}+\varepsilon){\bf 1}$ for every $t\ge t_{\varepsilon}$. Hence,

$$
\frac{dI_2}{dt}=d_2\mathcal{L}I_2+(S(t)-r_{2})\circ\beta_2\circ I_2\le_1 d_2\mathcal{L}I_2-(r_{2,\min}-r_{1,\min}-\varepsilon)\beta_{\min}I_2\quad \forall\ t\ge t_{\varepsilon},
$$
which implies that $I_2(t)\to{\bf 0}$ as $t\to\infty$, since $r_{2,\min}-r_{1,\min}-\varepsilon>0$. Finally, by the similar computations yielding \eqref{Z6}, we have that 
\begin{align*}
    \frac{d\widehat{I}_1}{dt}= d_1\mathcal{L}\widehat{I}_1+(\widehat{\beta_1\circ S\circ I_1}-\widehat{\gamma_1\circ I_1}) 
    =d_1\mathcal{L}\widehat{I}_1 +\widehat{Q}  
\end{align*}
where 
$$
Q=\beta_1\circ (S-r_{1,\min}{\bf 1})\circ I_1 +\beta_1\circ r_1\circ I_1-\gamma_1\circ I_1=\beta_1\circ (S-r_{1,\min}{\bf 1})\circ I_1. 
$$
It then follows from the variation of constant formula that
$$
\widehat{I}_1(t)=e^{d_1(t-\tau)\mathcal{L}}\widehat{I}_1(\tau)+\int_{0}^{t-\tau}e^{d_1(t-\tau -\sigma)\mathcal{L}}\widehat{Q}(\tau+\sigma)d\sigma\quad \forall\ t>\tau.
$$
We can now employ \eqref{Z3} and \eqref{Z8}, to derive that 
\begin{align*}
\|\widehat{I}_1(t)\|\le& e^{d_1(t-\tau)\tilde{d}_{*}}\|\widehat{I}_1(\tau)\|+\int_{0}e^{\tilde{d}_{*}d_1(t-\tau-\sigma)}\|\widehat{Q}(\tau+\sigma)\|d\sigma\cr 
\le & 2e^{d_1(t-\tau)\tilde{d}_{*}}\|{I}_1(\tau)\|+2\int_{0}e^{\tilde{d}_{*}d_1(t-\tau-\sigma)}\|{Q}(\tau+\sigma)\|d\sigma \cr 
=& 2e^{d_1(t-\tau)\tilde{d}_{*}}\|{I}_1(\tau)\|+2\int_{0}e^{\tilde{d}_{*}d_1(t-\tau-\sigma)}\|\beta_1\circ (S-r_{1,\min}{\bf 1})\circ I_1 (\tau+\sigma)\|d\sigma\cr 
\le & 2e^{d_1(t-\tau)\tilde{d}_{*}}\|{I}_1(\tau)\|+2\int_{0}e^{\tilde{d}_{*}d_1(t-\tau-\sigma)}\|\beta_1\|_{\infty}\|I_1(\tau+\sigma)\|_{\infty}\| (S-r_{1,\min}{\bf 1})(\tau+\sigma)\|d\sigma\cr 
\le & 2Ne^{\tilde{d}_*d_1(t-\tau)}+2\frac{N\|\beta_1\|_{\infty}}{d_1|\tilde{d}_*|}\sup_{\eta\ge \tau}\|S(\eta)-r_{1,\min}{\bf 1}\|.
\end{align*}
Hence, letting $t\to \infty$ and then $\tau\to\infty$ in the right hand side of this inequality and recalling that \eqref{U1-2}, we obtain that 
$$
\lim_{t\to\infty}\Big\|I_1(t)-\frac{\sum_{j\in\Omega}I_{1,j}}{k}{\bf 1}\Big\|=\lim_{t\to\infty}\|\widehat{I}_1(t)\|=0.
$$
Therefore, 
$$
\lim_{t\to\infty}\Big\|I_1(t)-\Big(\frac{N}{k}-r_{1,\min}\Big){\bf 1}\Big\|=0
$$
since  $\sum_{j\in\Omega}I_{1,j}=N-\sum_{j\in\Omega}S_j-\sum_{j\in\Omega}I_{2,j}\to N-kr_{1,\min}$ as $t\to\infty$. This completes the proof of {\rm (i)}.

\medskip

\noindent{\rm (ii)} If $\mathcal{R}_{0,1}(N)<\mathcal{R}_{0,2}(N)$, we know from Lemma \ref{lemA-3} that ${\bf E}_{1}^*$ is unstable. Now, suppose that $\mathfrak{R}_{1,\max}<\mathfrak{R}_{2,\min}$ (that is $r_{1,\min}>r_{2,\max}$) and we show that strain-1 eventually goes extinct.  Let $(S(t),I_1(t),I_2(t))$ be solution of \eqref{model-eq1} with a positive initial in $\mathcal{E}$. We claim that 
\begin{equation}\label{U2-1}
    \lim_{t\to\infty}\|I_1(t)\|_{\infty}=0.
\end{equation}
Suppose by contradiction that there is a sequence of positive numbers $\{t_n\}_{n\ge 1}$ converging to infinity such that 
\begin{equation*}
    \eta^*:=\inf_{n\ge 1}\|I_1(t_n)\|_{\infty}>0.
\end{equation*}
Hence, by \eqref{DP1}, there is $\tilde{\eta}_*>0$ such that 
\begin{equation}\label{U2-5}
    \inf_{n\ge 1}\min_{j\in\Omega}I_{1,j}(t_n)\ge \tilde{\eta}_*.
\end{equation}
By the Arzela-Ascoli theorem, if possible after passing to a subsequence, there is $(S^{\infty},I_1^{\infty},I_2^{\infty})\in C^1(\mathcal{R}:[\mathbb{R}_+^k]^3)$ such that $(S(t+t_{n}),I_1(t+t_{}),I_2(t+t_{n}))\to (S^{\infty}(t),I_1^{\infty}(t),I_2^{\infty}(t))$ as $n\to\infty$, for $t$ on compact subsets of $\mathbb{R}$. Furthermore, $ (S^{\infty},I_1^{\infty},I_2^{\infty})$ solves \eqref{U1-3}. By Lemma \ref{lemA-2}, there is $\eta_1>0$ such that
\begin{equation}\label{U2-2}
    I_{2,j}^{\infty}(t)\ge \eta_1\quad \forall\ t\in\mathbb{R},
\end{equation}
and by \eqref{T1-e2},
\begin{equation}\label{U2-3}
    S^{\infty}_{j}(t)\le S_{\max}=r_{1,\min}\quad \forall\ t\in\mathbb{R}, \ j\in\Omega.
\end{equation}
Moreover, by \eqref{U2-5}, we have that 
\begin{equation}\label{U2-4}
    I_{1,j}^{\infty}(0)\ge \tilde{\eta}_*\quad \forall\ j\in\Omega.
\end{equation}
From the second equation of \eqref{U1-3} and the fact that $\mathcal{L}$ is symmetric, we get
$$
\frac{d\sum_{j\in\Omega}I_{1,j}^{\infty}}{dt}=\sum_{j\in\Omega}(S^{\infty}_j-r_{1,\min})\beta_{1,j}I_{1,j}^{\infty}\quad \forall\ t\in\mathbb{R}.
$$
Integrating this equation yields
\begin{equation}
    \sum_{j\in\Omega}I^{\infty}_{1,j}(t)=\sum_{j\in\Omega}I_{1,j}^{\infty}(t-\tau)+\int_{t-\tau}^t\sum_{j\in\Omega}(S_j^{\infty}(\sigma)-r_{1,\min})\beta_{1,j}I_{1,j}^{\infty}(\sigma)d\sigma\quad \forall\ t\in\mathbb{R},\ \tau>0.
\end{equation}
Thus, in view of \eqref{U2-3}, we obtain  
\begin{equation}\label{U2-9}
    \int_{-\infty}^0\sum_{j\in\Omega}(r_{1,\min}-S^{\infty}_j(\sigma))\beta_{1,j}I_{1,j}^{\infty}(\sigma)d\sigma\le N
\end{equation}
and 
\begin{equation}\label{U2-10}
   \sum_{j\in\Omega}I^{\infty}_{1,j}(0)\le \sum_{j\in\Omega}I_{1,j}^{\infty}(t) \quad \forall\ t\le 0.
\end{equation}
From \eqref{U2-4} and \eqref{U2-9}, we get 
\begin{equation*}
    k\tilde{\eta}_*\le \sum_{j\in\Omega}I_{1,j}^{\infty}(t) \quad \forall\ t\le 0.
\end{equation*}
As a result, it follows from \eqref{DP1} that 
\begin{equation*}
    k\tilde{\eta}_*\le \sum_{j\in\Omega}I_{1,j}^{\infty}(t)\le k\|I_{1}^{\infty}(t)\|_{\infty}\le kc_*I_{1,\min}^{\infty}(t)\le c_*kI_{1,j}^{\infty}(t) \quad \forall\ t\le 0.
\end{equation*}
We then conclude from \eqref{U2-9}  and the last inequality that 
\begin{equation} \label{U2-12}
\int_{-\infty}^0\sum_{j\in\Omega}(r_{1,\min}-S^{\infty}_j(\sigma))d\sigma\le \frac{c_*}{\tilde{\eta}_*\beta_{\min}} \int_{-\infty}^0\sum_{j\in\Omega}(r_{1,\min}-S^{\infty}_j(\sigma))\beta_{1,j}I_{1,j}^{\infty}(\sigma)d\sigma\le \frac{N}{c_*\tilde{\eta}_*\beta_{\min}}.
\end{equation}
Next, from \eqref{U1-3}, we have that 
\begin{align}
    \frac{d\sum_{j\in\Omega}I_{2,j}^{\infty}}{dt}=& \sum_{j\in\Omega}(S^{\infty}_{j}-r_{1,\min})\beta_{2,j}I_{2,j}^{\infty}+\sum_{j\in\Omega}(r_{1,\min}-r_{2,j})\beta_{2,j}I_{2,j}^{\infty} \cr 
    \ge & \beta_{\max}\Big(\sum_{j\in\Omega}(S^{\infty}_{j}-r_{1,\min})\Big)\sum_{j\in\Omega}I_{2,j}^{\infty}+(r_{1,\min}-r_{2,\max})\sum_{j\in\Omega}\beta_{2,j}I_{2,j}^{\infty}
\end{align}
where we have used \eqref{U2-3}. An integration of the last inequality gives
\begin{equation*}
    \sum_{j\in{\Omega}}I_{2,j}^{\infty}(t)\ge e^{(r_{1,\min}-r_{2,\max})\tau -\beta_{\max}\int_{t-\tau}^t\sum_{j\in\Omega}(r_{1,\min}-S^{\infty}_{j}(\sigma))d\sigma}\sum_{j\in{\Omega}}I_{2,j}^{\infty}(t-\tau)\quad\ t\in\mathbb{R},\ \tau>0.
\end{equation*}
This along with \eqref{U2-2} and \eqref{U2-12} yield
\begin{align*}
    N\ge  \sum_{j\in{\Omega}}I_{2,j}^{\infty}(0)\ge& e^{(r_{1,\min}-r_{2,\max})\tau -\beta_{\max}\int_{-\tau}^0\sum_{j\in\Omega}(r_{1,\min}-S^{\infty}_{j}(\sigma))d\sigma}\sum_{j\in{\Omega}}I_{2,j}^{\infty}(-\tau)\cr 
    \ge &e^{(r_{1,\min}-r_{2,\max})\tau -\beta_{\max}\int_{-\infty}^0\sum_{j\in\Omega}(r_{1,\min}-S^{\infty}_{j}(\sigma))d\sigma}\sum_{j\in{\Omega}}I_{2,j}^{\infty}(-\tau)\cr 
    \ge &\eta_1e^{(r_{1,\min}-r_{2,\max})\tau -\frac{\beta_{\max}N}{c_*\tilde{\eta}_*\beta_{\min}}}\quad \forall\ \tau>0.
\end{align*}
    This is clearly impossible since $r_{1,\min}>r_{2,\max}$. Therefore, \eqref{U2-1} must hold. This together with \eqref{T1-e2} implies that $\liminf_{t\to\infty}\min_{j\in\Omega}I_{2,j}(t)\ge m_*$ for some positive number $m_*$ in dependent of initial data. This completes the proof of the theorem.
\end{proof}

Next, we give a proof of Theorem \ref{T3}.

\begin{proof}[Proof of Theorem \ref{T3}] We suppose that $\mathfrak{R}_{l,j}$ is constant in $j\in\Omega$ for each $l=1,2$. Hence $r_{l,\min}=r_{l,\max}$ for each $l=1,2$. Fix $l\ne p=1,2$. We distinguish three cases.

\noindent{\bf Case 1.} $\mathcal{R}_{0,l}(N)>\mathcal{R}_{0,p}(N)>1$.  In this case, it follows from Theorem \ref{T2} that the strain-l
 EE is globally stable with respect to positive initial data.

 \medskip
 
\noindent{\bf Case 2.} $\mathcal{R}_{0,l}(N)>1\ge \mathcal{R}_{0,p}(N)$. Without loss of generality, we suppose that $l=1$ and $p=2$. Let $(S(t),I_1(t),I_2(t))$ be solution of \eqref{model-eq1} with a positive initial data in $\mathcal{E}$ and define 
\begin{equation}\label{Lyapunov-function}
    \mathcal{L}(t)=\frac{1}{2}\sum_{j\in\Omega}S_j^2(t)+r_{1,\min}\sum_{j\in\Omega}I_{1,j}(t)+ r_{2,\min}\sum_{j\in\Omega}I_{2,j}(t)\quad \forall\ t\ge 0.
\end{equation}
Clearly, $\mathcal{L}>0$ for every $t>0$ and uniformly bounded above by $\frac{N^2}{2}+(r_{1,\max}+r_{2,\max})N$. Moreover, 
\begin{align*}
\frac{d\mathcal{L}}{dt}=&\sum_{j\in\Omega}\Big(d_S\sum_{i\in\Omega}L_{i,j}(S_i-S_j)+\sum_{l=1}^2(r_{l,\min}-S_j)\beta_{l,j}I_{l,j}\Big)S_j   +\sum_{l=1}^2 \sum_{j\in\Omega}r_{l,\min}(S_j-r_{l,\min})\beta_{l,j}I_{l,j}\cr
=&d_S\sum_{j\in\Omega}\sum_{i\in\Omega}(S_i-S_j)S_j+\sum_{l=1}^2\sum_{j\in\Omega}S_j(r_{l,\min}-S_{j})\beta_{l,j}I_{l,j}+\sum_{l=1}^2 \sum_{j\in\Omega}r_{l,\min}(S_j-r_{l,\min})\beta_{l,j}I_{l,j}\cr
=&-\frac{d_S}{2}\sum_{i,j\in\Omega}(S_i-S_j)^2+\sum_{l=1}^2 \sum_{j\in\Omega}(r_{l,\min}-S_j)(S_j-r_{l,\min})\beta_{l,j}I_{l,j}\cr
=&-\frac{d_S}{2}\sum_{i,j\in\Omega}(S_i-S_j)^2-\sum_{l=1}^2 \sum_{j\in\Omega}(S_j-r_{l,\min})^2\beta_{l,j}I_{l,j}\cr 
\le& 0.
\end{align*}
Therefore, $\mathcal{L}$ is a Lyapunov function. It then follows from the LaSalle's invariant principle (\cite[Theorem 4.3.4]{Henry}), the maximal invariant set, $\mathcal{I}$, of $\tilde{\mathcal{E}}:=\{ (S,I_1,I_2)\in\mathcal{E}\ :\ \frac{d\mathcal{L}}{dt}=0\}$ is the global attractor for solutions of \eqref{model-eq1}. Now, observe that $(S,I_1,I_2)\in\tilde{\mathcal{E}} $ if and only if 
\begin{equation}\label{U3-1}
    S_{j}=S_{1}, \quad \text{and}\quad (S_1-r_{l,\min})^2I_{l,j}=0\quad \forall\ l=1,2,\ j\in\Omega.
\end{equation}
As a result since $\mathcal{I}\subset \tilde{\mathcal{E}}$ and is invariant for the semiflow generated by solutions of \eqref{model-eq1} on $\mathcal{E}$, we conclude from \eqref{U3-1} that it is a subset of the spatially homogeneous EE solutions. Thus 
which is equivalent to 
\begin{equation*}
    \mathcal{I}\subset \Big\{\big(\frac{N}{k}{\bf 1},{\bf 0},{\bf 0}\big),\ (r_{1,\min}{\bf 1}, \big(\frac{N}{k}-r_{1,\min}\big){\bf 1},{\bf 0})\Big\}:=\{{\bf E}^0,{\bf E}_1^*\}.
\end{equation*}
Note here that \eqref{model-eq1} doesn't have a strain-2 constant EE solution since $\frac{N}{k}\le r_{2,\min}$. It is clear that $\{{\bf E}^0,{\bf E}^*_1\}\subset \tilde{\mathcal{E}}$ and is invariant for the semiflow generated by solution of \eqref{model-eq1}.  Therefore, $ \mathcal{I}=\{{\bf E}^0,{\bf E}_1^*\}$.  However, since $\mathcal{R}_0(N)=\mathcal{R}_{0,1}(N)>1$, the DFE solution ${\bf E}^0$ is a  repeller for solutions of \eqref{model-eq1} with positive initial data (see Lemma \ref{lem1}). Therefore, ${\bf E}_1^*$ is the global attractor for solution of \eqref{model-eq1}

\medskip

\noindent{\bf Case 3.} $1\ge \max\{\mathcal{R}_{0,1}(N),\mathcal{R}_{0,2}(N)\}$. In this, consider again  the Lyapunov function introduced in \eqref{Lyapunov-function}. This time we have that $\mathcal{I}=\{{\bf E}^0\}$. Hence the DFE attracts all solutions of \eqref{model-eq1} with initial data in $\mathcal{E}$.

\medskip Observe that {\rm (i)} of the theorem follows from case 3, while statement {\rm (ii)} follows from case 1 and case 2. 
\end{proof}

\subsection{Proofs of Theorems \ref{T4}, \ref{T5-1} and \ref{T5}}

\begin{proof}[Proof of Theorem \ref{T4}] Suppose that $d=d_1=d_2=d_S$. Let $(S(t),I_1(t),I_2(t))$ be solution of \eqref{model-eq1} with a positive initial data in $\mathcal{E}$. Then 
\begin{equation*}
\begin{cases}
    \frac{d(S+I_{1}+I_2)}{dt}=d\mathcal{L}(S+I_1+I_2) & t>0,\cr 
    N=\sum_{j\in\Omega}(S_j+I_{1,j}+I_{2,j}) & t\ge 0.
    \end{cases}
\end{equation*}
It then follows from \eqref{Z3} and \eqref{Z4} that 
\begin{equation*}
    \Big\|S(t)+I_1(t)+I_2(t)-\frac{N}{k}{\bf 1}\Big\|\le e^{t\tilde{d}_*}\|S(0)+I_{1}(0)+I_2(0)\|\le N\sqrt{k}e^{t\tilde{d}_*}\quad \forall\ t\ge 0. 
\end{equation*}
Therefore,
\begin{equation}\label{U3-2}
    \Big(\frac{N}{k}-N\sqrt{k}e^{t\tilde{d}_*}){\bf 1}-(I_1(t)+I_2(t)\Big)\le_1 S(t)\le_1 \Big (\frac{N}{k}+N\sqrt{k}e^{t\tilde{d}_*}\Big){\bf 1}-(I_1(t)+I_2(t))\quad \forall\ t\ge 0.
\end{equation}
As a result, 
\begin{equation}\label{U3-3}
    \frac{dI_l}{dt}\le_1 d_l\mathcal{L}I_l+ \Big(\beta_l\circ\Big(\Big(\frac{N}{k}+N\sqrt{k}e^{t\tilde{d}_*}\Big){\bf 1}-I_1-I_2\Big)-\gamma_l)\circ I_l.
\end{equation}
Now, suppose that $\mathcal{R}_{0,l}(N)\le 1$ for some $l=1,2$. This implies that $\lambda_l:=\lambda_*(d_l\mathcal{L}+{\rm diag}(\frac{N}{k}\beta_l-\gamma_l))\le 0$. Let $E_l$ be the positive eigenvector associated with $\lambda_l$ satisfying $\|E_{l}\|_{\infty}=1$. Observe that 

\begin{equation}\label{U3-4}
    \lambda_lE_{l}=d_l\mathcal{L}E_l+(\frac{N}{k}\beta_l-\gamma_l)\circ E_l.
\end{equation}
From \eqref{U3-3} and \eqref{U3-4}, we obtain that 
\begin{equation}
    \frac{d}{dt} E_l\circ I_l\le d_l(E_l\circ\mathcal{L}I_l-I_l\mathcal{L}E_l)- I_l\circ I_l\circ E_l +N\sqrt{k}e^{t\tilde{d}_*}\beta_l\circ E_l\circ I_l\quad \forall\ t>0.
\end{equation}
Since $\mathcal{L}$ is symmetric, it then follows that 
\begin{align}
    \frac{d}{dt}\sum_{j\in\Omega}E_{l,j}I_{l,j}\le & -\sum_{j\in\Omega}E_{l,j}I_{l,j}^2+N\sqrt{k}e^{t\tilde{d}_*}\beta_{\max}\sum_{j\in\Omega}E_{l,j}I_{l,j}\cr 
    \le & -\sum_{j\in\Omega}(E_{l,j}I_{l,j})^2+N\beta_{\max}\sqrt{k}e^{t\tilde{d}_*}\sum_{j\in\Omega}E_{l,j}I_{l,j} \quad (\text{since}\ \|E_{l}\|_{\infty}=1)\cr 
    \le & -\frac{1}{k}\Big(\sum_{j\in\Omega}E_{l,j}I_{l,j}\Big)^2++N\beta_{\max}\sqrt{k}e^{t\tilde{d}_*}\sum_{j\in\Omega}E_{l,j}I_{l,j} \quad (\text{by the Cauchy-Schwarz inequality})\cr 
    =& \Big(N\beta_{\max}\sqrt{k}e^{t\tilde{d}_*}-\frac{1}{k}\sum_{j\in\Omega}E_{l,j}I_{l,j}\Big)\sum_{j\in\Omega}E_{l,j}I_{l,j}.
\end{align}
Therefore, since $ \tilde{d}_*<0$, and hence $e^{t\tilde{d}_*}\to0$ as $t\to\infty$, we then conclude that $\sum_{j\in\Omega}E_{l,j}I_{l,j}(t)\to 0$ as $t\to\infty$. Consequently, $\|I_l(t)\|_{\infty}\to 0$ as $t\to\infty$ since $E_{l}$ is strictly positive. 

\quad {\rm (i)}  Now, if $\mathcal{R}_{0}(N)\le 1$, it follows from the above that $\|I_l(t)\|\to 0$ as $t\to\infty$, for every $l=1,2$. This together with \eqref{Z9}, yields that $\|S(t)-\frac{N}{k}{\bf 1}\|\to 0$ as $t\to\infty$. This shows that the DFE is globally stable if $\mathcal{R}_0(N)\le 1$. So, {\rm (i)} is proved.

\quad {\rm (ii)} Suppose without lost of generality that $\mathcal{R}_{0,1}(N)>1\ge \mathcal{R}_{0,2}(N)$. We know from the above that $\|I_2(t)\|\to 0$ as $t\to\infty$, which implies that $I_2(t)\pm N\sqrt{k}e^{t\tilde{d}_*}\to 0$ as $t\to\infty$. Recalling that $I_1(t)$ satisfies \eqref{U3-2} and \eqref{U3-3}. and 
\begin{align*}
 & d_1\mathcal{L}I_1+ \Big(\beta_1\circ\Big(\Big(\frac{N}{k}-N\sqrt{k}e^{t\tilde{d}_*}\Big){\bf 1}-I_1-I_2\Big)-\gamma_1)\circ I_1 \cr 
 \le_1 & \frac{dI_l}{dt} \le_1  d_1\mathcal{L}I_1+ \Big(\beta_1\circ\Big(\Big(\frac{N}{k}+N\sqrt{k}e^{t\tilde{d}_*}\Big){\bf 1}-I_1\Big)-\gamma_1)\circ I_1 \quad t>0,
\end{align*}
we can employ a perturbation argument to conclude that $I_1(t)\to I_1^*$ as $t\to\infty$, where $I^*$ is the unique positive solution of the multiple-patch logistic system
$$
0=d_1\mathcal{L}I_1^*+\Big(\beta_1\circ\Big(\frac{N}{k}{\bf 1}-I_1^*\Big)-\gamma_1\Big)I_1^*.
$$
The existence of $I_1^*$ follows from standard results on the logistic equations and the fact that $\lambda_*(d_1\mathcal{L}+{\rm diag}(\frac{N}{k}\beta_1-\gamma_1))>0$. Note that the last inequality holds since $\mathcal{R}_{0,1}(N)=\frac{N}{k}\rho(\mathcal{F}_1\mathcal{V}^{-1})>1$. We then conclude from \eqref{U3-2} that $(S(t),I_1(t),I_2(t))\to {\bf E}_1^*:=(\frac{N}{k}{\bf 1}-I_1^*,I_1^*,0)$ as $t\to\infty$. This completes the proof of {\rm (ii)}. 
\end{proof}

  \begin{proof}[Proof of Theorem \ref{T5-1}]Suppose that $k=2$, $\min\{\mathcal{R}_{0,1}(N),\mathcal{R}_{0,2}(N)\}>1$. Let $\tilde{\mathcal{R}}_1(N)$ and $\tilde{\mathcal{R}}_2(N)$ be defined by \eqref{GG2}.

  \quad {\rm (i)} Suppose that $\tilde{\mathcal{R}}_{1}(N)>1$. Let $(S(0),I_1(0),I_2(0))\in \mathcal{E}$ such that $\|I_1(0)\|_{\infty}>0$.
Take ${\bf E}^*:=(S^*,{\bf 0},I_2^*)\in \mathcal{E}_2^*$. Let ${\bf P}^*$ be an eigenvector, with positive entries, associated with  $\tilde{\mathcal{R}}_1({\bf E}_2^*)$, that is ${\bf E}_2^*$ satisfies 
  \begin{equation*}
      {\rm diag}((\beta_1\circ S^*_2)\mathcal{V}_1^{-1}){\bf P}^*=\tilde{\mathcal{R}}_1({\bf E}_2^*){\bf P^*}
  \end{equation*}
  which is equivalent to 
  $$
  0=d_1\mathcal{L}{\bf P}^*+\Big(\frac{1}{\tilde{\mathcal{R}}_1({\bf E}_2^*)}\beta_1\circ S_2^*-\gamma_1\Big)\circ{\bf P}^*
  $$
  Hence, since $\mathcal{L}$ is symmetric,
  \begin{align*}
      \frac{d}{dt}\sum_{j\in \Omega}P_j^*I_{1,j}(t)=& \sum_{j\in\Omega}\beta_{1,j}(S_j(t)-\frac{1}{\tilde{\mathcal{R}}_1({\bf E}_2^*)}S_{2,j}^*)P_{j}^*I_{1,j}(t)\cr 
      =&  \sum_{j\in\Omega}\beta_{1,j}(S_j(t)-S_{2,j}^*)P_{j}^*I_{1,j}(t)+ \frac{(\tilde{\mathcal{R}}_1({\bf E}_2^*)-1)}{\tilde{\mathcal{R}}_1({\bf E}_2^*)}\sum_{j\in\Omega}\beta_{1,j}S_{2,j}^*P_{j}^*I_{1,j}(t)\cr 
      \ge & \beta_{\max}\Big(\frac{\beta_{\min}}{\beta_{\max}}\frac{(\tilde{\mathcal{R}}_1(N)-1)}{\tilde{\mathcal{R}}_1(N)}-{\rm dist}((S(t),I_1(t),I_2(t)),\mathcal{E}_2^*)\Big)\sum_{j\in \Omega}P_j^*I_{1,j}(t)\quad t>0.
  \end{align*}
  So, an integration of this yields that 
  $$
  \sum_{j\in \Omega}P_j^*I_{1,j}(t)\ge e^{\int_0^t\beta_{\max}\Big(\frac{\beta_{\min}}{\beta_{\max}}\frac{(\tilde{\mathcal{R}}_1(N)-1)}{\tilde{\mathcal{R}}_1(N)}-{\rm dist}((S(\tau),I_1(\tau),I_2(\tau)),\mathcal{E}_2^*)\Big)d\tau}\sum_{j\in \Omega}P_j^*I_{1,j}(0)\quad t>0,
  $$
  from which it follows that 
  $$
  \frac{\ln\Big(\sum_{j\in\Omega}P^*_jI_{1,j}(t) \Big)}{t\beta_{\max}}\ge\frac{\beta_{\min}(\tilde{\mathcal{R}}_1(N)-1)}{\beta_{\max}\tilde{\mathcal{R}}_1(N)}-\frac{1}{t}\int_{0}^{t}{\rm dist}((S(\tau),I_1(\tau),I_2(\tau)),\mathcal{E}_2^*)d\tau +\frac{\ln\Big(\sum_{j\in\Omega}P_j^*I_{1,j}(0)\Big)}{\beta_{\max}t}.
  $$
  Since, $\frac{\ln(kN\|{\bf P}^*\|)}{\beta_{\max}t}\ge \frac{\ln\Big(\sum_{j\in\Omega}P^*_jI_{1,j}(t,j) \Big)}{t\beta_{\max}}$ for all $t>0$, letting $t\to\infty$ in the last inequality  implies that \eqref{2-GG1} holds. 

  Next, suppose in addition that $\mathcal{E}_2^*\cup\{{\bf E}^0\}$ is the global attractor for classical solutions of \eqref{model-eq1} with initial data $(S(0),I_1(0),I_2(0))\in\mathcal{E}$ satisfying $\|I_1(0)\|=0$. We appeal to the persistence theory \cite{ST2011} to prove \eqref{2-GG3}. To this end, define 
  \begin{equation}
      \mathcal{\xi}(S,I_1,I_2)=\|I_1\|_{\infty},\quad \quad (S,I_1,I_2)\in \mathcal{E}. 
  \end{equation}
  Clearly, $\mathcal{\xi}$ is continuous. Moreover, by the uniqueness of solution of \eqref{model-eq1}, we see that for any given initial data $(S(0),I_1(0),I_2(0))\in \mathcal{E}$,  $\mathcal{\xi}((S(t_0),I_1(t_0),I_2(t_0)))=0$ for some $t_0\ge 0$ if and only if $\mathcal{\xi}((S(t),I_1(t),I_2(t)))
  =0$ for all $t\ge 0.$  Since $\min\{\mathcal{R}_{0,1},\mathcal{R}_{0,2}\}>1$, by Lemma \ref{lem1}, we know that there is $\sigma_0>0$ such that 
  \begin{equation}\label{ZZX1}
      \limsup_{t\to\infty}\|(S(t),I_1(t),I_2(t))-{\bf E}^0\|\ge \sigma_0\quad \text{whenever}\quad \|I_1(0)\|+\|I_2(0)\|>0.
  \end{equation} 
   Note also from \eqref{2-GG1} there there is $\tilde{\sigma}_0>0$ such that 
  \begin{equation*}
      \limsup_{t\to\infty}{\rm dist}((S(t),I_1(t),I_2(t)),\mathcal{E}_2^*)\ge \tilde{\sigma}_0\quad \text{whenever}\quad \mathcal{\xi}((S(0),I_1(0),I_2(0)))>0.
  \end{equation*}
  Therefore,
 \begin{equation}\label{IST3}
      \limsup_{t\to\infty}{\rm dist}((S(t),I_1(t),I_2(t)),\mathcal{E}_2^*\cup\{{\bf E}^0\})\ge \min\{\sigma_0,\tilde{\sigma}_0\}\quad \text{whenever}\quad \mathcal{\xi}((S(0),I_1(0),I_2(0)))>0.
  \end{equation}

\noindent  Now, we claim that there is a positive $\sigma_1>0$ such that 
  \begin{equation}\label{last-claim-2}
      \limsup_{t\to\infty}\mathcal{\xi}(S(t),I_1(t),I_2(t))\ge \sigma_1\quad \text{whenever}\ (S(0),I_1(0),I_2(0))\in\mathcal{E}, \ \mathcal{\xi}(S(0),I_1(0),I_2(0))>0.
  \end{equation}
  We proceed by contradiction. So, suppose that there a sequence $\{(S^n(0),I_1^n(0),I_2^n(0))\}_{n\ge 1}$ of initial data in $\mathcal{E}$ satisfying $\mathcal{\xi}(S^n(0),I_1^n(0),I_2^n(0))>0$ and 
  \begin{equation}\label{IST1}
      \sup_{t\ge 0}\|I_1^n(t)\|_{\infty}=\sup_{t\ge 0}\mathcal{\xi}(S^n(t),I_1^n(t),I_2^n(t))\le \frac{1}{n}\quad \forall\ n\ge 1.
  \end{equation}
 By \eqref{T1-e1}, there is $\sigma_2>0$, such that, for each $n\ge 1$, there is $t_n\gg 1$ such that 
\begin{equation}\label{IST2}
    \min_{j\in\Omega}\sum_{l=1}^2I_{l,j}^n(t_n+t)\ge \sigma_2\quad \forall\ t\ge 0.
\end{equation}
It then follows from \eqref{IST1} that there is $n_0>1$
\begin{equation}\label{IST9}
    \min_{j\in\Omega}I_{2,j}^n(t_n+t)\ge \frac{\sigma_2}{2},\quad \forall\ t\ge 0, n\ge n_0.
\end{equation}
However, thanks to \eqref{IST3},  for each $n\ge 1$, there is $\tilde{t}_n>t_n$, $\tilde{t}_n-t_{n}\to \infty$, as $n\to\infty$, such that 
\begin{equation}\label{IST10}
    {\rm dist}((S^n(\tilde{t}_n),I_1^n(\tilde{t}_n),I_2(\tilde{t}_n)),\mathcal{E}_2^*\cup\{{\bf E}^0\})\ge\frac{1}{2}\min\{\tilde{\sigma}_0,\sigma_0\}. 
\end{equation}
Finally, consider the sequence of solutions $(S^n(t+\tilde{t}_n),I^n(t+\tilde{t}_n),I_2^n(t+\tilde{t}_n)))$ of \eqref{model-eq1}. By the Arzela-Ascoli theorem, possibly after passing to a subsequence, there is a nonnegative solution $(S^*(t),I^*_1(t),I_2^*(t))$ of \eqref{model-eq1}, defined for all $t\in\mathbb{R}$, such that $ (S^n(t+\tilde{t}_n),I^n(t+\tilde{t}_n,I_2^n(t+\tilde{t}_n)))\to(S^*(t),I^*_1(t),I_2^*(t)) $ as $n\to\infty$, locally uniformly on $\mathbb{R}$. Clearly, by \eqref{IST1}, $I_1^*(t)=0$ for all $t\in\mathbb{R}$. Moreover, we claim that 
\begin{equation}\label{IST21}
    (S^*(t),{\bf 0},I_2^*(t))\in \mathcal{E}_2^*\quad \forall\ t\in\mathbb{R}.
\end{equation}
  To see this, note from \eqref{IST9} and the fact that $\tilde{t}_n-t_n\to\infty$ as $n\to\infty$ that $I_2^*(t)\ge \frac{\sigma_2}{2}$ for all $t\in\mathbb{R}$. This also implies that $\|(S^*(t),0,I_2^*(t))-{\bf E}^0\|\ge \frac{\sigma_2}{2}$ for all $t\in\mathbb{R}$. Therefore, since $\mathcal{E}_2^*\cup\{{\bf E}^0\}$ is the global attractor for solutions of \eqref{model-eq1}, when restricted to initial data $I_1(0)=0$, and \eqref{ZZX1} holds, it follows from \cite[Theorem 5.7]{ST2011}  that \eqref{IST21} holds. As a result, we obtain that  ${\rm dist}((S^*(0),0,I_2^*(0)),\mathcal{E}_2^*\cap\{{\bf E}^0\})=0$, which contradicts with \eqref{IST10}. Therefore, claim \eqref{last-claim-2} holds. We can now invoke  persistence theory \cite[Theorem 5.2]{ST2011}, we deduce that there is $\sigma_{1,*}>0$ such that 
   \begin{equation*}
      \liminf_{t\to\infty}\mathcal{\xi}(S(t),I_1(t),I_2(t))\ge \sigma_1\quad \text{whenever}\ (S(0),I_1(0),I_2(0))\in\mathcal{E}, \ \mathcal{\xi}(S(0),I_1(0),I_2(0))>0.
  \end{equation*}
  Therefore, by the similar arguments as those in the proof of step 2 of Theorem \ref{T1}-{\rm (ii-3)} derive that \eqref{2-GG3} holds.

  \medskip

  {\rm (ii)} Suppose that desired hypotheses hold. Define the mapping $\tilde{\xi}\ :\ \mathcal{E}\to[0,\infty)$ by 
  $$
  \tilde{\xi}(S,I_1,I_2)=\min_{j\in\Omega}\min_{l=1,2}I_{l,j}.
  $$
  Then $\tilde{\xi}$ is continuous and concave. Note that $\tilde{\xi}(S(t),I_1(t),I_2(t))>0$ for all $t>0$ if $\tilde{\xi}(S(0),I_1(0),I_2(0))>0$.  Furthermore, by \eqref{ZZX1}, there is $\nu_*>0$ such that 
  $$
  \liminf_{t\to\infty}\tilde{\xi}(S(t),I_1(t),I_2(t))\ge\nu_*\quad \text{whenever}\quad \tilde{\xi}(S(0),I_1(0),I_2(0))>0.
  $$
  Therefore, since the set $\mathcal{E}$ is convex and compact, and the semiflow induced by the solutions of \eqref{model-eq1} is also compact and continuous, we can invoke \cite[Theorem 6.2]{ST2011} to conclude that system \eqref{model-eq1} has at least one coexistence EE solution.
  \end{proof}

We conclude this section with a proof of Theorem \ref{T5}.

\begin{proof}[Proof of Theorem \ref{T5}] Suppose that {\bf (A2)} holds and $\min\big\{\mathcal{R}_{0,1}(N),\mathcal{R}_{0,2}(N)\big\}>1$. Let ${\bf E}_1^*=(\frac{N}{k}{\bf 1}-I^*_1,I_1^*,{\bf 0})$ and ${\bf E}_2^*=(\frac{N}{k}{\bf 1}-I_2^*,{\bf 0},I_2^*)$ be the single-strain EE solutions of \eqref{model-eq1}.  It follows from \cite[Theorem 3]{LP2023} that the single strain EE solution are globally stable with respect to positive initial on each of the set $\mathcal{E}_1$ and $\mathcal{E}_2$. Hence, the assertions {\rm (i)} and {\rm (ii)} follows from Theorem \ref{T5-1}.

\quad {\rm (iii)} By \cite[Theorem 2.7]{CSSY2020}, $\mathcal{R}_{0,l}(N)\to \mathfrak{R}_{l,\max}(N)$ as $d\to0$, for each $i=1,2$. Hence, since for each $j=1,2$, $H^+_{l}\ne\emptyset$, there is $d_{*,1}>0$ such that $\min\{\mathcal{R}_{0,1}(N),\mathcal{R}_{0,2}(N)\}>1$ for every $0<d<d_{*,1}$. This ensures that the unique single-strain EE solutions, ${\bf E}_1^*=(\frac{N}{k}{\bf 1}-I^*_1,I_1^*,{\bf 0})$ and ${\bf E}_2^*=(\frac{N}{k}{\bf 1}-I_2^*,{\bf 0},I_2^*)$, exist whenever $0<d_{*,1}$. Moreover, by \cite[Theorem 9-{\rm b}]{LP2023}, it holds that 
$$
I_l^*\to \Big(\frac{k}{N}{\bf 1}-r_{l}\Big)_+\quad \text{as}\ d\to 0,\quad l=1,2.
$$
 Therefore, by \cite[Theorem 2.7]{CSSY2020}, it holds that
\begin{equation}\label{HJH}
    \tilde{\mathcal{R}}_l(N)=\rho\Big(\Big(\frac{N}{k}\mathcal{F}_l-{\rm diag}(\beta_l\circ I_p^*)\Big)\mathcal{V}_l^{-1}\Big)\to\tilde{\mathcal{R}}_l^*(N):=\max_{j\in\Omega}\mathfrak{R}_{l,j}\min\Big\{\frac{N}{k},r_{p,j}\Big\} \quad \text{as}\quad  d\to 0, \quad p\ne l\in\{1,2\}.
\end{equation}
However, for each $l\ne p\in\{1,2\}$ and    $j\in\Sigma_{l}\cap H^{+}_l\ne\emptyset$ , it holds that $\frac{N}{k}\mathfrak{R}_{l,j}>1$ and $\mathfrak{R}_{l,j}r_{p,j}>1$, that is, $\mathfrak{R}_{l,j}\min\Big\{\frac{N}{k},r_{p,l}\Big\}>1$. Therefore, $\min\Big\{\tilde{\mathcal{R}}_{1}^*(N),\tilde{\mathcal{R}}_2^*(N) \Big\}>1$. Hence, by \eqref{HJH}, there is $0<d_*<d_{*,1}$ such that $\min\{\tilde{\mathcal{R}}_1(N),\tilde{\mathcal{R}}_2(N)\}>1$ for every $0<d<d_*$.
\end{proof}

\end{document}